\documentclass[11pt]{article}
\usepackage{amsfonts}
\usepackage{amsmath,amssymb}
\usepackage{amsthm}
\usepackage[mathscr]{euscript}
\usepackage{mathrsfs}
\usepackage{indentfirst}
\usepackage{color}\usepackage{enumitem}

\newtheorem{theorem}{Theorem}[section]
\newtheorem{lemma}[theorem]{Lemma}
\newtheorem{corollary}[theorem]{Corollary}

\newtheorem{remark}[theorem]{Remark}
\newtheorem{proposition}[theorem]{Proposition}
\numberwithin{equation}{section}

\newcommand{\lbl}[1]{\label{#1}}
\allowdisplaybreaks

\newcommand{\be}{\begin{equation}}
\newcommand{\ee}{\end{equation}}
\newcommand\bes{\begin{eqnarray}} \newcommand\ees{\end{eqnarray}}
\newcommand{\bess}{\begin{eqnarray*}}
\newcommand{\eess}{\end{eqnarray*}}
\newcommand{\bbbb}{\left\{\begin{aligned}}
\newcommand{\nnnn}{\end{aligned}\right.}
\newcommand{\bea}{\begin{align*}}
\newcommand{\eea}{\end{align*}}

\newcommand\ep{\varepsilon}

\newcommand\kk{\left}
\newcommand\rr{\right}
\newcommand\dd{\displaystyle}

\newcommand\dx{{\rm d}x}
\newcommand\dy{{\rm d}y}

\newcommand\yy{\infty}

\newcommand\sk{\smallskip}

\newcommand\ccc{\color{blue}}

\begin{document}\thispagestyle{empty}
\setlength{\baselineskip}{16pt}

\begin{center}
 {\LARGE\bf Sharp estimates for a nonlocal diffusion problem with a free boundary\footnote{This work was supported by NSFC Grants
11771110, 11971128}}\\[4mm]
  {\Large Lei Li, \ \  Mingxin Wang\footnote{Corresponding author. {\sl E-mail}: mxwang@hit.edu.cn}}\\[1.5mm]
{School of Mathematics, Harbin Institute of Technology, Harbin 150001, PR China}
\end{center}

\date{\today}

\begin{abstract}
In this paper, we proceed to study the nonlocal diffusion problem proposed by Li and Wang \cite{LW21}, where the left boundary is fixed, while the right boundary is a nonlocal free boundary. We first give some accurate estimates on the longtime behavior by constructing lower solutions, and then investigate the limiting profiles of this problem when the expanding coefficient of free boundary converges to $0$ and $\yy$, respectively. At last, we focus on two important kinds of kernel functions, one of which is compactly supported and the other behaves like $|x|^{-\gamma}$ with $\gamma\in(1,2]$ near infinity. With the help of some upper and lower solutions, we obtain some sharp estimates on the longtime behavior and rate of accelerated spreading.



\end{abstract}

\section{Introduction}
Recently, in \cite{CDLL} the following nonlocal diffusion model with free boundaries has been investigated
\bes\left\{\begin{aligned}\label{1.1}
&u_t=d\int_{g(t)}^{h(t)}J(x,y)u(t,y)\dy-du+f(u), & &t>0,~g(t)<x<h(t),\\
&u(t,x)=0,& &t>0, ~ x\notin(g(t),h(t)),\\
&h'(t)=\mu\int_{g(t)}^{h(t)}\int_{h(t)}^{\infty}
J(x,y)u(t,x)\dy\dx,& &t>0,\\
&g'(t)=-\mu\int_{g(t)}^{h(t)}\int_{-\infty}^{g(t)}
J(x,y)u(t,x)\dy\dx,& &t>0,\\
&h(0)=-g(0)=h_0>0,\;\; u(0,x)=\tilde u_0(x),& &|x|\le h_0,
 \end{aligned}\right.
 \ees
where $J(x,y)=J(x-y)$, $\tilde{u}_0\in C([-h_0,h_0])$, $\tilde{u}_0(\pm h_0)=0<\tilde{u}_0$ in $(-h_0,h_0)$, the kernel $J$ satisfies
 \begin{enumerate}[leftmargin=4em]
\item[{\bf(J)}]$J\in C(\mathbb{R})\cap L^{\yy}(\mathbb{R})$, $J\ge 0$, $J(0)>0,~\dd\int_{\mathbb{R}}J(x)\dx=1$, \ $J$\; is even,
 \end{enumerate}
 and the growth term $f$ satisfies
  \begin{enumerate}
  \item[{\bf(F)}]\, $f\in C^1(\mathbb{R})$, $f(0)=0<f'(0)$, $f(u^*)=0>f'(u^*)$ for some $u^*>0$, and $\frac{f(u)}{u}$ is strictly decreasing in $u>0$.
 \end{enumerate}

The authors proved that \eqref{1.1} has a unique global solution $(u,g,h)$, and established the spreading-vanishing dichotomy. In particular, they applied the ODE theory and the contracting mapping principle to obtain the existence and uniqueness of global solution. Then Du, Li and Zhou \cite{DLZ} discussed the spreading speed of \eqref{1.1} when spreading happens. More precisely, they found that there exists a threshold condition on $J$, i.e.,
\begin{enumerate}[leftmargin=4em]
\item[{\bf(J1)}]$\dd\int_{0}^{\yy}xJ(x)\dx<\yy$,
 \end{enumerate}
such that \eqref{1.1} has a finite spreading speed if and only if {\bf(J1)} holds true. Moreover, they considered the limiting profile of \eqref{1.1} when the expanding rate $\mu$ of free boundary converges to $\yy$. Afterwards, Du and Ni \cite{DN21} not only extended the above results to monostable cooperative systems with partially degenerate diffusion and free boundaries, but gave more comprehensive and delicate conclusions on spreading speed. Especially, if $J$ satisfies $J(x)\approx |x|^{-\gamma}$ with $\gamma>1$, namely,
\begin{enumerate}[leftmargin=4em]
\item[${\bf(J^\gamma)}$] there exist $\varsigma_1,\varsigma_2>0$ such that $\varsigma_1|x|^{-\gamma}\le J(x)\le \varsigma_2|x|^{-\gamma}$ when $|x|\gg 1$,
 \end{enumerate}
then they obtained a complete understanding of spreading speed. Besides, Du and Ni \cite{DN212} introduced  the high dimensional and radial symmetry version of \eqref{1.1}, and some difficulties caused by kernel function have been overcome by a series of new methods. For finite spreading speed and accelerated spreading, a new threshold condition on kernel function was found. There are other recent works on nonlocal diffusion problem with free boundaries, please see e.g. \cite{DWZ,DN200,ZZLD} and references therein.

Very recently, Li and Wang \cite{LW21} put forward the following model
 \bes\label{1.2}
\left\{\begin{aligned}
&u_t=d\dd\int_{0}^{h(t)}J(x,y)u(t,y)\dy-d\dd\left(\int_{0}^{\yy}J(x,y)\dy\right) u+f(u), && t>0,~0\le x<h(t),\\
&u(t,h(t))=0, && t>0,\\
&h'(t)=\mu\dd\int_{0}^{h(t)}\int_{h(t)}^{\infty}
J(x,y)u(t,x)\dy\dx, && t>0,\\
&h(0)=h_0,\;\; u(0,x)=u_0(x), &&x\in[0,h_0],
\end{aligned}\right.
 \ees
 where $J$ satisfies {\bf(J)}, and $u_0\in C([0,h_0])$, $u_0(h_0)=0<u_0(x)$ in $[0,h_0)$.

In this model, the species is assumed not to jump to the domain $(-\yy,0)$, which implies that the species only expands their habitat through the right boundary. The existence and uniqueness of global solution of \eqref{1.2} was first proved by using the similar arguments with those of \cite{CDLL}. Then the spreading-vanishing dichotomy was established, namely, one of the following alternatives must hold for \eqref{1.2}:

\sk{\rm(1)}\, \underline{Spreading:} $\lim_{t\to\yy}h(t)=\yy$ and $\lim_{t\to\yy}u=u^*$ locally uniformly in $\overline{\mathbb{R}}^+$;

\sk{\rm(2)}\, \underline{Vanishing:} $\lim_{t\to\yy}h(t)<\yy$ and $\lim_{t\to\yy}\|u(t,\cdot)\|_{C([0,h(t)])}=0$.

In addition, similarly to \cite{DLZ}, the authors also derived that if spreading happens for \eqref{1.2}, then
 \begin{align*}
\lim_{t\to\yy}\frac{h(t)}{t}=\left\{\begin{aligned}
&c_0& &{\rm if~{\bf(J1)}~is~satisfied},\\
&\yy& &{\rm if~{\bf(J1)}~is~not~satisfied},
\end{aligned}\right.
  \end{align*}
  where $c_0$ is uniquely given by the semi-wave problem \eqref{2.1}.

Here we would like to remark that the problems with the Stefan boundary condition have been studied extensively and systematically since the pioneering work \cite{DL2010}. For example, interested readers may refer to \cite{GLZ,DL,Wang,DDL,HHD} and references therein for some recent developments on such problems. Particularly, by virtue of some subtle upper and lower solutions, Du, Matsuzawa and Zhou \cite{DMZ} obtained some sharp estimates on the solution of problem
 \bes\label{1.3}
\left\{\begin{aligned}
&u_t=du_{xx}+\hat f(u), && t>0,~g(t)<x<h(t),\\
&u(t,g(t))=u(t,h(t))=0, && t>0,\\
&g'(t)=-\mu u_x(t,g(t)), ~ h'(t)=-\mu u_x(t,h(t)), && t>0,\\
&-g(0)=h(0)=h_0,\;\; u(0,x)=\hat u_0(x), &&x\in[-h_0,h_0],
\end{aligned}\right.
 \ees
 where $\hat{u}_0(x)\in C^2([-h_0,h_0])$, $\hat{u}_0(\pm h_0)=0<\hat{u}_0(x)$ in $(-h_0,h_0)$, $\hat{u}'(-h_0)>0>\hat{u}'(h_0)$. The function $\hat f$ is of monostable, bistable, or combustion type, and thus has a unique maximal positive zero $\hat{u}^*$. To be exact, they proved that if spreading happens, then the solution of \eqref{1.3} satisfies
 \bes\label{1.4}
 \left\{\begin{aligned}
 &\lim_{t\to\yy}h'(t)=k_0, ~ ~\;\; \lim_{t\to\yy}\big(h(t)-k_0t-\hat{H}\big)=0 ~ ~ {\rm for ~ some} ~ \hat{H}\in\mathbb{R},\\
 &\lim_{t\to\yy}g'(t)=-k_0, ~ ~ \lim_{t\to\yy}\big(g(t)+k_0t-\hat{G}\big)=0 ~ ~ {\rm for ~ some} ~ \hat{G}\in\mathbb{R},\\
 &\lim_{t\to\yy}\max_{x\in[0,h(t)]}\big|u(t,x)-q_{k_0}(h(t)-x)\big|=0,\ \
 \lim_{t\to\yy}\max_{x\in[g(t),0]}\big|u(t,x)-q_{k_0}(x-g(t))\big|=0,
 \end{aligned}\right.
 \ees
 where $(k_0,q_{k_0})$ is uniquely given by the semiwave problem
  \bes\label{1.5}
 \left\{\begin{aligned}
 &dq''-kq'+\hat f(q)=0, ~ ~ q(x)>0, ~~ 0<x<\yy,\\
 &q(0)=0, ~ q(\yy)=\hat u^*, ~\; \mu q'(0)=k.
 \end{aligned}\right.
 \ees

Motivated by the above works, in this paper we continue to study problem \eqref{1.2}. In Section 2, we first give some more accurate estimates for longtime behavior of solution of \eqref{1.2}, and then analyze the limiting profiles of \eqref{1.2} when $\mu\to0$ and $\mu\to\yy$, respectively. Section 3 is devoted to the discussion of spreading speed of \eqref{1.2} when $J$ has a compact support. In Section 4, we assume that $J(x)\approx |x|^{-\gamma}$ with $\gamma\in(1,2]$, and give estimates on the rate of accelerated spreading and longtime behavior for \eqref{1.2}.

\section{Limiting profiles of \eqref{1.2} as $\mu\to 0$ and $\mu\to\yy$ }
\setcounter{equation}{0} {\setlength\arraycolsep{2pt}

For convenience, some notations are given here. Denote $\mathbb{R}^+=(0,\yy)$, $\overline{\mathbb{R}}^+=[0,\yy)$ and $j(x)=\int_{0}^{\yy}J(x,y)\dy$.
In this paper, we always assume that $J$ satisfies the condition {\bf(J)}.

\begin{proposition}[{\cite[Theorem 1.2]{DLZ}}]\label{p2.1}\, Let $f$ satisfy the condition {\bf(F)}. Then the problem
\bes\label{2.1}\left\{\begin{array}{lll}
 d\dd\int_{-\yy}^{0}J(x,y)\phi(y)\dy-d\phi+c\phi'+f(\phi)=0, \quad -\yy<x<0,\\[1mm]
\phi(-\yy)=u^*,\ \ \phi(0)=0, \ \ c=\mu\dd\int_{-\yy}^{0}\int_{0}^{\yy}J(x,y)\phi(x)\dy\dx
 \end{array}\right.
 \ees
 has a unique solution pair $(c_0,\phi^{c_0})$ with $c_0>0$ and $\phi^{c_0}(x)$ nonincreasing in $(-\yy,0]$ if and only if {\bf(J1)} is satisfied. We usually call $\phi^{c_0}$ the semi-wave solution of \eqref{2.1} with speed $c_0$.
\end{proposition}
\begin{theorem}\label{t2.1}Let $(u,h)$ be a solution of \eqref{1.2} and spreading happen. Then we have
\bess\left\{\begin{aligned}
&\lim_{t\to\yy}\max_{x\in[0,\,ct]}|u(t,x)-u^*|=0 ~ {\rm for ~ any ~ } c\in(0,c_0) ~ ~ \text{if the condition {\bf(J1)} holds},\\
&\lim_{t\to\yy}\max_{x\in[0,\,ct]}|u(t,x)-u^*|=0 ~ {\rm for ~ any ~ } c>0 ~ ~ \text{if the condition {\bf(J1)} does not hold}.
\end{aligned}\right.\eess
\end{theorem}
\begin{proof} By a simple comparison argument, we have $\limsup_{t\to\yy}u(t,x)\le u^*$ uniformly in $\overline{\mathbb{R}}^+$.
Therefore, it suffices to show the lower limit of $u$. We first prove the conclusion with {\bf(J1)} being satisfied. It is easy to show that there is a small $\delta_0>0$ such that $\tilde{f}(u):=-\delta u+f(u)$ satisfies {\bf(F)} for all $\delta\in(0,\delta_0)$, and has a unique positive zero $u^*_{\delta}$. Obviously, $u^*_{\delta}\to u^*$ as $\delta\to0$. It then follows from \cite[Theorem 3.15]{LW21} that the semi-wave problem
 \bess\left\{\begin{array}{lll}
 d\dd\int_{-\yy}^{0}J(x,y)\phi(y)\dy-d\phi(x)+c\phi'(x)+\tilde f(\phi(x))=0,\;\;  -\yy<x<0,\\[3mm]
\phi(-\yy)=u^*_{\delta},\ \ \phi(0)=0, \ \ c=\mu\dd\int_{-\yy}^{0}\int_{0}^{\yy}J(x,y)\phi(x)\dy\dx
 \end{array}\right.
 \eess
 has a unique solution pair $(c^{\delta}_0,\phi^{\delta}_0)$ with $\phi^{\delta}_0$ nonincreasing and $c^{\delta}_0 \to c_0$ as $\delta\to0^+$. Thus, for any $c\in(0,c_0)$, we can find a $\delta_1$ with $0<\delta_1<\delta_0$ such that $c^{\delta}_0>c$ for all $\delta\in(0,\delta_1)$.
Define
\[\xi(x)=1, ~ |x|\le1; ~ ~ ~  \xi(x)=2-|x|, ~ 1\le|x|\le2; ~ ~ ~ \xi(x)=0, ~ |x|\ge2,\]
and $J_n(x)=\xi(\frac{x}{n})J(x)$. Clearly, $J_n$ are supported compactly and nondecreasing in $n$, $J_n(x)\le J(x)$,
\[J_n(x)\to J(x) ~ {\rm ~in~} L^1(\mathbb{R}){\rm~and~locally~uniformly~in~\mathbb{R}},\]
and
\[j_n(x):=\int_{-x}^{\yy}J_n(y)\dy\to j(x)=\int_{-x}^{\yy}J(y)\dy ~ ~ {\rm uniformly~in~}\mathbb{R}{\rm ~as~} n\to\yy.\]
 For any $\delta\in(0,\delta_1)$, we can choose $n$ large enough, say $n\ge N>0$, such that $d(\|J_n\|_{L^1}-1)u+\tilde{f}(u)$ still meets {\bf(F)} and $d(j_n(x)-j(x))+\delta\ge0$ for all $x\in \mathbb{R}$.

For any given $n\ge N$, let $(u_n,h_n)$ be the unique solution of
 \bess
\left\{\begin{aligned}
&u_{nt}=d\dd\int_{0}^{h_{n}(t)}J_n(x,y)u_n(t,y)\dy-dj_n(x)u_n+\tilde{f}(u_n), && t>0,~0\le x<h_n(t),\\
&u_n(t,h_n(t))=0, && t>0,\\
&h_n'(t)=\mu\dd\int_{0}^{h_n(t)}\!\!\int_{h_n(t)}^{\infty}
J_n(x,y)u_n(t,x)\dy\dx, && t>0,\\
&u_n(0,x)=u(T,x),~h_n(0)=h(T), \ \ x\in[0,h(T)].
\end{aligned}\right.
 \eess
By the comparison principle (\cite[Theorem 3.7]{LW21}), we have
\[u(t+T,x)\ge u_n(t,x), ~ ~ h(t+T)\ge h_n(t) ~ ~ {\rm for ~ }t\ge0, ~x\in[0,h_n(t)].\]
Moreover, it can be seen from the proof of \cite[Theorem 3.15]{LW21} that the following problem
\bess\left\{\begin{array}{lll}
 d\dd\int_{-\yy}^{0}J_n(x,y)\phi_n(y)\dy-d\phi_n(x)+c\phi_n'(x)+\tilde f(\phi_n(x))=0, \;\; -\yy<x<0,\\[3mm]
\phi_n(-\yy)=u^*_{n,\delta},\ \ \phi_n(0)=0, \ \ c_n=\mu\dd\int_{-\yy}^{0}\int_{0}^{\yy}J_n(x,y)\phi_n(x)\dy\dx
 \end{array}\right.
 \eess
 has a unique solution pair $(c^{\delta}_n,\phi^{\delta}_n)$ with $c^{\delta}_n>0$ and $\phi^{\delta}_n$ nonincreasing in $(-\yy,0]$. Here $u^*_{n,\delta}$ is the unique positive root of the equation $d(\|J_n\|_{L^1}-1)u+\tilde{f}(u)=0$. Clearly, $\lim_{n\to\yy}u^*_{n,\delta}= u^*_{\delta}$. By \cite[Lemma 4.5]{LW21},  we have $c^{\delta}_n\nearrow c^{\delta}_0$ as $n\to\yy$.
For any small $\ep>0$ and some $L\gg1$, we define
\[\underline{h}(t)=c^{\delta}_n(1-\ep)t+2L, ~ ~ \underline{u}(t,x)=(1-\ep)\phi^{\delta}_n(x-\underline{h}(t)).\]
As in the proof of \cite[Theorem 3.15]{LW21}, there exist suitable $T,\,T_1>0$ such that
\[u_n(t+T_1,x)\ge\underline{u}(t,x), ~ ~ h_n(t+T_1)\ge \underline{h}(t) ~ ~ {\rm for ~ }t\ge0, ~x\in[0,\underline{h}(t)].\]
Hence,
\[u(t+T+T_1,x)\ge\underline{u}(t,x), ~ ~ h(t+T+T_1)\ge \underline{h}(t) ~ ~ {\rm for ~ }t\ge0, ~x\in[0,\underline{h}(t)].\]
On the other hand, we easily derive
\[\max_{x\in[0,(1-2\ep)c^{\delta}_nt]}\big|\underline{u}(t,x)-
(1-\ep)u^*_{n,\delta}\big|=(1-\ep)\big[u^*_{n,\delta}-\phi^{\delta}_n(-\ep c^{\delta}_nt-2L)\big]\to 0\]
as $t\to\yy$. So we have $\liminf_{t\to\yy}u(t,x)\ge (1-\ep)u^*_{n,\delta}$ uniformly in $[0,(1-2\ep)c^{\delta}_nt]$. For any $c\in (0,c_0)$, we can take $\ep$ small enough and $N$ large sufficiently such that $c<(1-2\ep)c^{\delta}_n$. Therefore,
$\liminf_{t\to\yy}u(t,x)\ge (1-\ep)u^*_{n,\delta}$ uniformly in $[0,ct]$. The arbitrariness of $\ep,\delta$ and $n$ leads to
\[\liminf_{t\to\yy}u(t,x)\ge u^* ~ ~ {\rm uniformly ~ in ~}[0,ct].\]
Thus we get the desired result.

As for the case with {\bf(J1)} being violated, it follows from \cite[Proposition 5.1]{DN212} that $c^{\delta}_n\to\yy$ as $n\to\yy$. We then easily derive that for any $c>0$,
\[\liminf_{t\to\yy}u(t,x)\ge u^* ~ ~ {\rm uniformly ~ in ~}[0,ct].\]
The proof is complete.
\end{proof}

\begin{remark}\label{r2.1}Assume that spreading occurs for \eqref{1.2}. If {\bf(J1)} holds true, then we have that, for any $c\in(0,c_0)$,
\[\lim_{t\to\yy}\max_{x\in[0,\,ct]}\big|u(t,x)-\phi^{c_0}(x-h(t))\big|=0.\]
In fact, since $\lim_{t\to\yy}h(t)/t=c_0>c$, one easily deduces
  \bess\lim_{t\to\yy}\max_{x\in[0,\,ct]}\big|u(t,x)-\phi^{c_0}(x-h(t))\big|
  &\le&\lim_{t\to\yy}\max_{x\in[0,\,ct]}\big|u(t,x)-u^*\big|+\lim_{t\to\yy}
  \max_{x\in[0,\,ct]}\big|u^*-\phi^{c_0}(x-h(t))\big|\\
&=&\lim_{t\to\yy}\max_{x\in[0,\,ct]}\big|u(t,x)-u^*\big|+\lim_{t\to\yy}
\left[u^*-\phi^{c_0}(ct-h(t))\right]\\
&=&0.
  \eess
\end{remark}

\begin{corollary}\label{c2.1} Let $(u,g,h)$ be the unique solution of \eqref{1.1}. If spreading happens, then
\bess\left\{\begin{aligned}
&\lim_{t\to\yy}\max_{x\in[-ct,\,ct]}\big|u(t,x)-u^*\big|=0 ~ {\rm for ~ any ~ } c\in(0,c_0) ~ ~ {\rm  ~ if ~ {\bf(J1)} ~ is ~ satisfied },\\
&\lim_{t\to\yy}\max_{x\in[-ct,\,ct]}\big|u(t,x)-u^*\big|=0 ~ {\rm for ~ any ~ } c>0 ~ ~ {\rm  ~ if ~ {\bf(J1)} ~ is ~ violated}.
\end{aligned}\right.\eess
\end{corollary}
\begin{proof}We just give the sketch of the proof since it can be obtained by similar methods with that of Theorem \ref{t2.1}.
From the proof of \cite[Lemma 3.2]{DLZ}, for any small $\ep>0$ and some $\ell\gg1$, there exists $T>0$ such that
\bess
 &u(t+T,x)\ge (1-\ep)\left[\phi^{c_0}(x-\sigma(t))+\phi^{c_0}(-x-\sigma(t))-u^*\right],&\\[1mm]
 &h(t+T)\ge\sigma(t), ~ ~ g(t+T)\le-\sigma(t) ~ ~ {\rm for } ~ t\ge0, ~ x\in[-\sigma(t),\sigma(t)],&
 \eess
where $\sigma(t):=(1-2\ep)c_0t+\ell$. Direct calculations show that, for $x\in[-(1-3\ep)c_0t,(1-3\ep)c_0t]$,
\[|\phi^{c_0}(x-\sigma(t))+\phi^{c_0}(-x-\sigma(t))-2u^*|=2[u^*-\phi^{c_0}(-\ep c_0t-\ell)]\to 0\]
as $t\to\yy$. Thus, we have
\[\liminf_{t\to\yy}u(t,x)\ge (1-\ep)u^* ~ {\rm uniformly ~ in} ~ [-(1-3\ep)c_0t,(1-3\ep)c_0t].\]
For any $c\in(0,c_0)$, one may choose $\ep$ small enough, say $\ep\in(0,\ep_1)$, such that $c<(1-3\ep)c_0$ for any $\ep\in(0,\ep_1)$. Due to the arbitrariness of $\ep$, we obtain
$\liminf_{t\to\yy}u(t,x)\ge u^*$ uniformly in $[-ct,ct]$, which, together with the fact that $\limsup_{t\to\yy}u(t,x)\le u^*$ uniformly in $[g(t),h(t)]$, yields the assertion when $J$ satisfies {\bf(J1)}.

For the result with {\bf(J1)} being violated, we may argue as in the proof of Theorem \ref{t2.1} and the arguments in \cite[Section 4]{DLZ} to prove it. The details are omitted here.
\end{proof}
Now we are going to consider the limiting profiles of \eqref{1.2} as $\mu\to0$ and $\mu\to\yy$, respectively. For the sake of discussion, we first state a proposition which can be found in \cite{DLZ,Ya}.
 \begin{enumerate}[leftmargin=4em]
\item[{\bf(J2)}] There exists $\lambda>0$ such that
\[\dd\int_{-\yy}^{\yy}J(x)e^{\lambda x}\dx<\yy.\]
 \end{enumerate}

\begin{proposition}[{\cite[Proposition 1.3]{DLZ}}]\label{p2.2} Let the condition {\bf(J)} hold and $f$ satisfy the condition {\bf(F)}. If the condition {\bf(J2)} holds, then there exists a constant $c_*>0$ such that the problem
 \bes\label{2.2}\left\{\begin{array}{lll}
 d\dd\int_{-\yy}^{\yy}J(x,y)\phi(y)\dy-d\phi+c\phi'+f(\phi)=0, \quad x\in\mathbb{R},\\[1mm]
\phi(-\yy)=u^*,\ \ \phi(\yy)=0
 \end{array}\right.
 \ees
 has a nonincreasing solution $\phi_c$ with speed $c$ if and only if $c\ge c_*$. And for $c\ge c_*$, $\phi_c\in C^1(\mathbb{R})$. If $J$ does not satisfy the condition {\bf(J2)}, then \eqref{2.2} does not have such nonincreasing solution.
 \end{proposition}

Furthermore, we have that the solution $\phi_c$ of \eqref{2.2} is positive and strictly decreasing in $\mathbb{R}$. In fact, if $\phi_c$ is equal to $0$ somewhere, then there are $x_0\in\mathbb{R}$ and small $\sigma>0$ such that $\phi_c(x_0)=0<\phi_c(x)$ in $(x_0-\sigma,x_0)$. Thanks to \eqref{2.2} and the condition {\bf(J)}, we see
\[0=\int_{-\yy}^{\yy}J(x_0,y)\phi(y)\dy=\int_{-\yy}^{\yy}J(y)\phi(y+x_0)\dy>0.\]
This contradiction implies $\phi_c(x)>0$ in $\mathbb{R}$. We next show the monotonicity of $\phi_c$. Clearly, there exist $a\in\mathbb{R}$ and $\delta_0>0$ such that $J(x)>0$ for $x\in[-\delta_0,\delta_0]$ and $\phi_c(x)$ is strictly decreasing in $[a,a+\delta_0]$.  To our aim, it is sufficient to prove that $\phi_c(x)$ is strictly decreasing in $[a+n\delta_0,a+(n+1)\delta_0]$ for any integer $n$. We only prove the case with $n\ge0$ since the other case can be handled by analogous arguments. Obviously, the conclusion holds for $n=0$, and by induction we assume that $\phi_c(x)$ is strictly decreasing in $[a+n\delta_0,a+(n+1)\delta_0]$. Assume that there exist $x_1$, $x_2$: $a+(n+1)\delta_0\le x_1<x_2\le a+(n+2)\delta_0$ such that $\phi_c(x_1)=\phi_c(x_2)$. Without loss of generality, we may suppose $\phi'_c(x_1)=\phi'_c(x_2)=0$. In view of \eqref{2.2}, we have
 \bess
0=\int_{-\yy}^{\yy}J(y)\left(\phi(y+x_1)-\phi_c(y+x_2)\right)\dy
\ge\int_{-\delta_0}^{\delta_0}J(y)\left(\phi(y+x_1)-\phi_c(y+x_2)\right)\dy>0
 \eess
by the continuities of $J$ and $\phi_c$ as well as the facts $J(-\delta_0)>0$ and $\phi_c(x_1-\delta_0)-\phi_c(x_2-\delta_0)>0$. This contradiction shows that $\phi_c$ is strictly decreasing in $[a+(n+1)\delta_0,a+(n+2)\delta_0]$.

 Consider the cauchy problem
 \bes\left\{\begin{aligned}\label{2.3}
&w_t= d\int_{-\yy}^{\yy}J(x,y)w(t,y)\dy-dw+f(w),\;\; t>0,\ x\in\mathbb{R},\\
&w(0,x)=\tilde u_0(x),  ~ |x|\le h_0; ~ ~ w(0,x)\equiv0, ~ |x|>h_0.
 \end{aligned}\right.\ees
It is well known that this problem has a unique global solution $w$, and $\lim_{t\to\yy}w(t,x)=u^*$ locally uniformly in $\mathbb{R}$.
We can study the spreading behavior of the problem \eqref{2.3} by discussing the level set
 \[E_{\lambda}(t):=\left\{x\in\mathbb{R}: w(t,x)=\lambda\right\} ~ ~ {\rm for ~ any ~ given ~}\lambda\in(0,u^*).\]
More precisely, define $x^+_{\lambda}(t):=\sup E_{\lambda}(t)$ and $x^-_{\lambda}(t):=\inf E_{\lambda}(t)$. If the condition {\bf(J2)} holds, one may have (see e.g. \cite{Wei}) that $\lim_{t\to\yy}\frac{|x^{\pm}_{\lambda}(t)|}{t}=c_*$ which is well known as the spreading speed of \eqref{2.3}; If the condition {\bf(J2)} does not hold, then $\lim_{t\to\yy}\frac{|x^{\pm}_{\lambda}(t)|}{t}=\yy$ which is usually called the accelerated spreading for \eqref{2.3}.

The authors of \cite{DLZ} proved that problem \eqref{2.3} is the limiting problem of \eqref{1.1} as $\mu\to\yy$. To stress the dependence of solution pair of \eqref{2.1} on $\mu$, we denote by $(c_{\mu},\phi_{\mu})$ the unique solution pair of \eqref{2.1}. Due to the properties of $\phi_{\mu}$, there is a unique $l_{\mu}>0$ such that $\phi_{\mu}(-l_{\mu})=u^*/2$. The limiting behavior of $(c_{\mu},\phi_{\mu})$ as $\mu\to\yy$ was also obtained in \cite{DLZ}.

\begin{proposition}[{\cite[Theorem 5.1 and Theorem 5.2]{DLZ}}]\label{p2.3} Suppose that the condition {\bf(F)} holds. If $J$ satisfies the condition {\bf(J2)}, then
\[c_{\mu}\to c_*, ~ ~ l_{\mu}\to\yy, ~ ~ \phi_{\mu}(x)\to0 ~ {\rm and } ~ \phi_{\mu}(x-l_{\mu})\to\phi_*(x) ~ ~ {\rm locally ~ uniformly ~ in} ~ \mathbb{R} ~ {\rm as} ~ \mu\to\yy,\]
where $(c_*,\phi_*)$ is the minimal speed solution pair of \eqref{2.2} with $\phi_*(0)=u^*/2$; If the condition {\bf(J1)} holds but the condition {\bf(J2)} does not hold, then $\lim_{\mu\to\yy}c_{\mu}=\yy$.
\end{proposition}

This proposition shows that if {\bf(J2)} holds, the spreading speed of free boundary of \eqref{1.1} converges to the spreading speed of \eqref{2.3} as $\mu\to\yy$. Now we give the behavior of $(c_{\mu},\phi_{\mu})$ as $\mu\to0$.
\begin{theorem}\label{t2.2} Suppose that the conditions {\bf(F)} and {\bf(J1)} hold. Then
\bess
&&\dd\lim_{\mu\to 0}c_{\mu}=0, ~ ~ \lim_{\mu\to 0}\phi_{\mu}(x)\to U(x) ~  {\rm locally ~ uniformly ~ in} ~ (-\yy,0), \\
&&\left\{\begin{aligned}
&\dd\lim_{\mu\to 0}l_{\mu}=0& &{\rm if} ~ U(0)\ge {u^*}/2,\\
&\dd\lim_{\mu\to 0}l_{\mu}=x_0& &{\rm if} ~ U(0)< {u^*}/2,  ~ {\rm with } ~ x_0<0 ~ {\rm satisfying}\;  U(x_0)={u^*}/2,
\end{aligned}\right.\\
&&\dd\lim_{\mu\to 0}\phi_{\mu}(x-l_{\mu})=U(x-x_0) ~ ~ {\rm locally ~ uniformly ~ in} ~ (-\yy,x_0),
\eess
where $U(x)$ is the unique bounded positive solution of
 \bess
  d\dd\int_{-\yy}^{0}J(x,y)U(y)\dy-dU+f(U)=0 {\rm ~ ~ ~ in ~ }\;(-\yy,0].
  \eess
\end{theorem}
\begin{proof}Without loss of generality, we choose a sequence $\mu_n\to0$ and rewrite $(c_{\mu_n},\phi_{\mu_n},l_{\mu_n})$ as $(c_n,\phi_n,l_n)$. Obviously,
\[0\le c_n\le\mu_n u^*\int_{-\yy}^{0}\int_{0}^{\yy}J(x,y)\dy\dx\to0 ~ ~ ~ {\rm as } ~ n\to\yy.\]
Since $\phi'_n<0$ and $0\le\phi_n\le u^*$, by Helly's theorem, there are a nonincreasing function $\phi_{\yy}$ with $0\le\phi_{\yy}\le u^*$ and a subsequence of $\phi_n$, still denoted by itself, such that $\phi_n\to\phi_{\yy}$ almost everywhere in $(-\yy,0]$ and $\phi_{\yy}(0)=0$. For any $x<0$, it can be seen from \eqref{2.1} that
\bess
c_n\phi_n(x)-c_n\phi_n(0)=-\int_{0}^{x}\left(d\dd\int_{-\yy}^{0}\!J(z,y)\phi_n(y)\dy-d\phi_n(z)+f(\phi_n(z))\right){\rm d}z.
\eess
By virtue of the dominated convergence theorem, we have
\[0=-\int_{0}^{x}\left(d\dd\int_{-\yy}^{0}\!J(z,y)
 \phi_\yy(y)\dy-d\phi_\yy(z)+ f(\phi_\yy(z))\right){\rm d}z.\]
 Thus, differentiate the above equation to yield
 \bes\label{2.p}
 d\dd\int_{-\yy}^{0}\!J(x,y)\phi_\yy(y)\dy-d\phi_\yy+f(\phi_\yy)=0 ~ ~ {\rm in~ } (-\yy,0).
 \ees
We here claim that there are only two possible cases for $\phi_\yy$, namely, Case 1: $\phi_{\yy}\equiv0$ in $(-\yy,0]$, and Case 2: $\phi_{\yy}(x)=U(x)$ for $x<0$.
 Since $\phi_{\yy}$ is nonincreasing, the discontinuous points of $\phi_{\yy}$ are almost countable. If $\phi_{\yy}(x)=0$ for all continuous points of $\phi_{\yy}$, then $\phi_{\yy}\equiv0$ in $(-\yy,0]$. Otherwise, we can define
 \[x_*=\sup\{x<0: \phi_{\yy} ~ {\rm is ~ continuous ~ in ~} x ~ {\rm and ~ } \phi_{\yy}(x)>0\}.\]
If $x_*<0$, then $\phi_{\yy}\equiv0$ in $(x_*,0]$. Taking $x\to x_*^+$ in \eqref{2.p} leads to
 $\int_{-\yy}^{0}\!J(x_*,y)\phi_\yy(y)\dy=0$,
 which contradicts to the definition of $x_*$. Hence $x_*=0$. Since $\phi_{\yy}$ is bounded and nonincreasing in $(-\yy,0]$, by letting $x\to0^-$ in \eqref{2.p} we have
 \[d\dd\int_{-\yy}^{0}\!J(y)\phi_\yy(y)\dy-d\phi_\yy(0^-)+f(\phi_\yy(0^-))=0\]
with $\phi_{\yy}(0^-):=\lim_{x\to0^-}\phi_\yy(x)$. Due to the definition of $x_*$, conditions {\bf(F)} and {\bf(J)}, we obtain $\phi_{\yy}(0^-)>0$. This together with the monotonicity of $\phi_\yy$ yields that $\phi_{\yy}(x)\ge\phi_{\yy}(0^-)$ for $x<0$. Define $\hat \phi_{\yy}(x)=\phi_{\yy}(x)$ for $x<0$ and $\hat \phi_{\yy}(0)=\phi_{\yy}(0^-)$. Clearly, $\hat{\phi}_{\yy}$ satisfies \eqref{2.p} in $(-\yy,0]$ and $\phi_{\yy}(0^-)\le\hat{\phi}_{\yy}\le u^*$. By \cite[Lemma 2.4]{LW21}, one immediately derives $\hat{\phi}_{\yy}\equiv U$ in $(-\yy,0]$. Thus our claim is proved.

Now we show that Case 1 cannot happen. Define $\tilde{\phi}_n(x)=\phi_n(x-l_n)$ for $x\le l_n$ and $\tilde{\phi}_n(x)=0$ for $x>l_n$. By Helly's theorem, there exist a nonincreasing function $\tilde \phi_{\yy}$ with $0\le\tilde\phi_{\yy}\le u^*$ and a subsequence of $\tilde \phi_n$, still denoted by itself, such that $\tilde \phi_n\to\tilde\phi_{\yy}$ almost everywhere in $\mathbb{R}$ and $\tilde\phi_{\yy}(0)=u^*/2$. Moreover, it is easy to see from the monotonicity of $\phi_n$ and $\phi_{\yy}\equiv0$ that $\phi_n(x)\to\phi_{\yy}(x)$ locally uniformly in $(-\yy,0]$, which implies $l_n\to\yy$. Therefore, for any $x\in\mathbb{R}$, it follows from \eqref{2.1} that for large $n$,
 \bess
 c_n\tilde\phi_n(0)-c_n\tilde\phi_n(x)=-\int_{0}^{x}\left(d\dd\int_{-\yy}^{l_n}\!J(z,y)\tilde\phi_n(y)\dy-d\tilde\phi_n(z)+f(\tilde\phi_n(z))\right){\rm d}z.
 \eess
Using  the dominated convergence theorem, we have
\[\int_{0}^{x}\left(d\dd\int_{-\yy}^{\yy}\!J(z,y)\tilde\phi_\yy(y)\dy-d\tilde\phi_\yy(z)+f(\tilde\phi_\yy(z))\right){\rm d}z=0,\]
which leads to
\[d\dd\int_{-\yy}^{\yy}\!J(x,y)\tilde\phi_\yy(y)\dy-d\tilde\phi_\yy(x)+f(\tilde\phi_\yy(x))=0 ~ ~ ~ {\rm in }~ ~ \mathbb{R}.\]
Since $\tilde\phi_\yy$ is nonincreasing, $0\le\tilde\phi_{\yy}\le u^*$ and $\tilde\phi_{\yy}(0)=u^*/2$, we easily show that $\tilde\phi_\yy(-\yy)=u^*$ and $\tilde\phi_\yy(\yy)=0$, which implies that \eqref{2.2} has a solution $\tilde{\phi}_\yy$ with speed $c=0$. Clearly, this is a contradiction to Proposition \ref{p2.2}. Thus Case 1  can not occur.

It follows from the monotonicity of $\phi_n$ and continuity of $U$ that $\phi_n\to U$ locally uniformly in $(-\yy,0)$ as $n\to\yy$. Now we investigate the limit of $l_n$. We first show that $U$ is strictly decreasing in $(-\yy,0]$. In fact, it follows from \cite[Lemma 2.4]{LW21} that $U$ is nonincreasing in $(-\yy,0]$. By the condition  {\bf(J)}, there is a $\delta_0>0$ such that $J>0$ in $[-\delta_0,\delta_0]$. As before, it suffices to show that $U$ is strictly decreasing in $[-(k+1)\delta_0,-k\delta_0]$ for each integer $k\ge0$. If there exist $-\delta_0\le x_1<x_2\le 0$ such that $U(x_1)=U(x_2)$. Then we have, from the equation satisfied by $U$,
\bess
\int_{-\yy}^{-x_2}\!J(y)U(y+x_2)\dy=\left\{\int_{-x_2}^{-x_1}+\int_{-\yy}^{-x_2}\right\}
J(y)U(y+x_1)\dy>\int_{-\yy}^{-x_2}\!J(y)U(y+x_2)\dy.
\eess
This contradiction implies that $U$ is strictly decreasing in $[-\delta_0,0]$. Arguing inductively, assume that $U$ is strictly decreasing in $[-(k+1)\delta_0,-k\delta_0]$. If there exist $-(k+2)\delta_0\le x_1<x_2\le-(k+1)\delta_0$ such that $U(x_1)=U(x_2)$. Similarly to the above, we have
\bess
0&=&\int_{-\yy}^{-x_1}\!J(y)U(y+x_1)\dy-\int_{-\yy}^{-x_2}\!J(y)U(y+x_2)\dy\\
 &\ge&\int_{-\yy}^{-x_2}\!J(y)\left(U(y+x_1)-U(y+x_2)\right)\dy\\
 &\ge&\int_{-\delta_0}^{\delta_0}\!J(y)\left(U(y+x_1)-U(y+x_2)\right)\dy>0
\eess
since $J(\delta_0)>0$ and $U(x_1+\delta_0)-U(x_2+\delta_0)>0$. This is a contradiction.

Now we deal with the case $U(0)\ge u^*/2$. Assume that $\lim_{n\to\yy}l_n\ne0$. Then there are $\ep_0\in(0,\delta_0)$ and a subsequence of $l_n$, still denoted by itself, such that $l_n\ge\ep_0>0$. Thus, $u^*/2=\phi_n(-l_n)\ge\phi_n(-\ep_0)\to U(-\ep_0)>U(0)\ge u^*/2$ as $n\to\yy$.  This is a contradiction. For the case $U(0)<u^*/2$, clearly, there is a unique $x_0<0$ such that $U(x_0)=u^*/2$. If the assertion is not true, one can find a small $\ep>0$ and a subsequence of $l_n$, still denoted by itself, such that $l_n\le x_0-\ep$ or $l_n\ge x_0+\ep$. We only discuss the former case since their proofs are similar. By monotonicity, $u^*/2=\phi_n(-l_n)\le\phi_n(-x_0+\ep)\to U(-x_0+\ep)<u^*/2$ as $n\to\yy$. We get a contradiction.

We may argue as in the above analysis to show the limit of $\phi_n(x-l_n)$, and the details are omitted here. By the arbitrariness of sequence $\mu_n$, the proof is finished.
\end{proof}

\begin{remark}\label{r2.3}We claim that $U(0)>u^*/2$ when $d$ is small, and $U(0)<u^*/2$ when $d$ is large. Actually, we can examine the behaviors of $U$ as $d$ approaches $0$ and $\yy$, respectively. To stress the dependence of $U$ on $d$, we rewrite $U$ as $U_d$. It follows from the equation of $U_d$ and $0<U_d(x)<u^*$ in $(-\yy,0]$ that for any $0<d_1<d_2$,
\[ d_2\dd\int_{-\yy}^{0}J(x,y)U_{d_1}(y)\dy-d_2U_{d_1}+f(U_{d_1})<0 {\rm ~ ~ ~ in ~ }\;(-\yy,0].\]
Then we can obtain $U_{d_1}(x)\ge U_{d_2}(x)$ for $x\le0$ by using similar methods in the proof of {\rm\cite[Lemma 2.4]{LW21}}. Combining with $0<U_d(x)<u^*$, we can define $U_0(x):=\lim_{d\to0^+}U_{d}(x)$ and $U_\yy(x):=\lim_{d\to\yy}U_{d}(x)$ for $x\le0$. Clearly, $0\le U_\yy(x)\le U_0(x)\le u^*$. By the equation of $U_d$, it is easy to see that $U_\yy(x)\equiv0$ and $U_0(x)\equiv u^*$ in $(-\yy,0]$. Thus our claim holds.
 \end{remark}

 \begin{remark}\label{r2.4}We remark that when $\hat f(q)$ takes the form $q(a-bq)$ with $a,b>0$, the solution pair $(k_0,q_{k_0})$ of \eqref{1.5} shares the analogous behaviors with those in Proposition \ref{p2.3} and Theorem \ref{t2.2} as $\mu$ converges to $0$ and $\yy$, respectively. Denote the unique solution pair of \eqref{1.5} with $\hat f(q)=q(a-bq)$ by $(k_{\mu},q_{\mu})$. Then by \cite[Proposition 3.1]{DG} and some simple analysis, one has the following results:

\sk{\rm(1)}\, $\lim_{\mu\to\yy}k_{\mu}=2\sqrt{ad}$, $\lim_{\mu\to\yy}\ell_{\mu}=\yy$, and  $\lim_{\mu\to\yy}q_{\mu}(x)=0$ in $C^2_{loc}(\mathbb{R}^+)$, $\lim_{\mu\to\yy}q_{\mu}(x+\ell_{\mu})=q_{\yy}(x)$ in $C^2_{loc}(\mathbb{R})$, where $\ell_{\mu}>0$ is uniquely determined by $q_{\mu}(\ell_{\mu})=\frac{a}{2b}$, and $q_{\yy}$ is a solution of
    \bess
 dq''-2\sqrt{ad} q'+q(a-bq)=0, ~ ~ x\in\mathbb{R}; ~ ~ ~ q(-\yy)=0, ~ q(\yy)=\frac{a}{b}, ~ ~ q'(x)\ge0 ~ {\rm for} ~ x\in\mathbb{R};
 \eess

\sk{\rm(2)}\, $\lim_{\mu\to0}k_{\mu}=\lim_{\mu\to0}\ell_{\mu}=0$, $\lim_{\mu\to0}q_{\mu}(x)=p(x)$ in $C^2_{loc}(\mathbb{R}^+)$, where $p(x)$ is the unique positive solution of
  \bess
 &dp''+p(a-bp)=0, ~ ~ x>0; ~ ~ ~ p(0)=0.
 \eess
 \end{remark}

Consider the problem
 \bes\left\{\begin{aligned}\label{2.4}
&V_t= d\int_{0}^{\yy}J(x,y)\big(V(t,y)-V(t,x)\big)\dy+f(V), \;\; t>0,\ x\in \overline{\mathbb{R}}^+,\\
&V(0,x)=u_0(x), ~ x\in[0,h_0]; ~ ~ ~ V(0,x)\equiv 0, ~  x\ge h_0,
 \end{aligned}\right.\ees
and define $\tilde{E}_{\lambda}(t):=\big\{x\in\overline{\mathbb{R}}^+: V(t,x)=\lambda\big\}$, $\tilde {x}^+_{\lambda}(t)=\sup \tilde{E}_{\lambda}(t)$ and $\tilde{x}^-_{\lambda}(t)=\inf \tilde{E}_{\lambda}(t)$ for $\lambda\in(0,u^*)$. We now show that \eqref{2.4} has the same spreading speed with \eqref{2.3}.

\begin{proposition}\label{p2.4} Suppose that $f$ satisfies the condition {\bf(F)}. We have the following conclusions.

\sk{\rm(1)}\, If the condition {\bf(J2)} is satisfied, then $\lim_{t\to\yy}\frac{\tilde{x}^{\pm}_{\lambda}(t)}{t}=c_*$.

\sk{\rm(2)}\, If the condition {\bf(J2)} is violated, then $\lim_{t\to\yy}\frac{\tilde{x}^{\pm}_{\lambda}(t)}{t}=\yy$.
 \end{proposition}

\begin{proof} \, {\rm(1)}\, Since the condition {\bf(J2)} holds, by Proposition \ref{p2.2}, the problem \eqref{2.2} has a solution $\phi_{c_*}$ with speed $c_*$. Note that $\phi_{c_*}>0$ in $\mathbb{R}$ and $\phi_{c_*}\in C^1(\mathbb{R})$.  We can find a $K\ge1$ such that $K\phi_{c_*}(x)\ge V(0,x)$ for $x\in\overline{\mathbb{R}}^+$. From the properties of $\phi_{c_*}$, one may see that for any $\lambda\in(0,u^*)$, there exists the unique $y_0$, depending only on $K$ and $\lambda$, such that $K\phi_{c_*}(y_0)=\lambda$. Let $\bar{V}(t,x)=K\phi_{c_*}(x-c_*t)$. Then for $x\in\overline{\mathbb{R}}^+$, we have
 \bess
 \bar{V}_t(t,x)&=&-c_*K\phi'_{c_*}(x-c_*t)\\
 &\ge& d\int_{-\yy}^{\yy}J(x,y)\bar V(t,y)\dy-d\bar V(t,x)+f(\bar V)\\
 &\ge&d\int_{0}^{\yy}J(x,y)\bar V(t,y)\dy-dj(x)\bar V(t,x)+f(\bar V).
 \eess
Since $K\phi_{c_*}(x)\ge, \not\equiv V(0,x)$ in $\overline{\mathbb{R}}^+$. By the maximum principle (\cite[Lemma 2.2]{LW21}), $\bar{V}(t,x)>V(t,x)$ for $t>0$ and $x\in\overline{\mathbb{R}}^+$. Thus, for $t\gg1$, we have $\tilde {x}^+_{\lambda}(t)\le y_0+c_*t$,
 which yields
 \[\limsup_{t\to\yy}\frac{\tilde {x}^+_{\lambda}(t)}{t}\le c_* ~ ~ {\rm and } ~ ~ \limsup_{t\to\yy}\frac{\tilde {x}^-_{\lambda}(t)}{t}\le c_*.\]

It now remains to show
 \bess\liminf_{t\to\yy}\frac{\tilde {x}^+_{\lambda}(t)}{t}\ge c_* ~ ~ {\rm and } ~ ~ \liminf_{t\to\yy}\frac{\tilde {x}^-_{\lambda}(t)}{t}\ge c_*.
 \eess
To stress the dependence on $\mu$, we denote the unique solution of \eqref{1.2} by $(u_{\mu},h_{\mu})$. Since {\bf(J2)} indicates {\bf(J1)}, we obtain that $\lim_{t\to\yy}h_{\mu}(t)/t=c_{\mu}$. Moreover, by Proposition \ref{p2.3}, we have that $c_{\mu}\to c_*$ as $\mu\to\yy$. By the maximum principle (\cite[Lemma 2.1]{LW21}), one can see that for any $\mu>0$, $V(t,x)\ge u_{\mu}(t,x)$ for $t\ge0$ and $x\in[0,h_{\mu}(t)]$. For $\lambda\in(0,u^*)$, we choose $\delta>0$ small such that $\lambda<u^*-\delta$. Then from Theorem \ref{t2.1}, one has that for any $0<\ep\ll1$, there exists $T>0$ so that
 \[u^*-\delta\le u_{\mu}(t,x)\le u^*+\delta ~ ~ {\rm for} ~ t\ge T, ~ x\in[0,(c_{\mu}-\ep)t],\]
which yields $\tilde {x}^+_{\lambda}(t)\ge (c_{\mu}-\ep)t$ and $\tilde {x}^-_{\lambda}(t)\ge (c_{\mu}-\ep)t$ for $t\gg1$. Letting $\ep\to0$ and $\mu\to\yy$, we obtain the assertion (1).

\sk{\rm(2)}\, If {\bf(J2)} is violated but {\bf(J1)} holds, from the above analysis we have $\tilde {x}^+_{\lambda}(t)\ge (c_{\mu}-\ep)t$ and $\tilde {x}^-_{\lambda}(t)\ge (c_{\mu}-\ep)t$ for $t\gg1$. By Proposition \ref{p2.3}, $\lim_{\mu\to\yy}c_{\mu}=\yy$, and thus the desired result holds.

If {\bf(J1)} does not hold, by Theorem \ref{t2.1} and choosing $\delta>0$ as above one may see that for any $c>0$, there is $T_1>0$ such that
   \[u^*-\delta\le u_{\mu}(t,x)\le u^*+\delta ~ ~ {\rm for} ~ t\ge T_1, ~ x\in[0,ct],\]
which implies $\tilde {x}^+_{\lambda}(t)\ge ct$ and $\tilde {x}^-_{\lambda}(t)\ge ct$ for $t\gg1$. Letting $c\to\yy$, we get the assertion (2).
\end{proof}

From Propositions \ref{p2.3}, \ref{p2.4} and \cite[Theorem 4.4]{LW21}, we see that if the condition {\bf(J2)} holds, then the spreading speed of \eqref{1.2} converges to that of \eqref{2.4} as $\mu\to\yy$. Our next conclusion shows that problem \eqref{2.4} is the limiting problem of \eqref{1.2} as $\mu\to\yy$. Moreover, it is well known that problem
  \bes\left\{\begin{aligned}\label{2.5}
&v_t= d\int_{0}^{h_0}J(x,y)v(t,y)\dy-dj(x)v(t,x)+f(v), && t>0,\ x\in [0,h_0],\\
&v(0,x)=u_0(x), && x\in[0,h_0]
 \end{aligned}\right.\ees
 has a unique positive solution $v\in C(\overline{\mathbb{R}}^+\times[0,h_0])$. The following theorem also indicates that problem \eqref{2.5} can be viewed as the limiting problem of \eqref{1.2} as $\mu\to0$.

 \begin{theorem}\label{t2.3}Suppose that $f$ satisfies the condition {\bf(F)}. Let $(u_{\mu},h_{\mu})$ be a solution of \eqref{1.2} and $V,\,v$ be solutions of \eqref{2.4} and \eqref{2.5}, respectively. Then the followings hold:

 \sk{\rm(1)}\, $u_{\mu}\to v$ and $u_{\mu,t}\to v_t$ in $C_{loc}(\mathbb{R}^+\times[0,h_0])$, $h_{\mu}\to h_0$ in $C^1_{ loc}(\mathbb{R}^+)$ as $\mu\to0$;

  \sk{\rm(2)}\, $u_{\mu}\to V$ in $C_{loc}(\mathbb{R}^+\times\mathbb{R}^+)$, $h_{\mu}\to \yy$ locally uniformly in $\mathbb{R}^+$ as $\mu\to\yy$.
 \end{theorem}

 \begin{proof} \sk{\rm(1)}\, Firstly, comparison principles assert that $(u_{\mu},h_{\mu})$ is increasing in $\mu>0$, $u_{\mu}(t,x)\ge v(t,x)$ and $h_{\mu}(t)\ge h_0$ for $t\ge0$ and $x\in[0,h_0]$. Thus we can define $v_0(t,x)=\lim_{\mu\to0}u_{\mu}(t,x)$ as well as $h_*(t)=\lim_{\mu\to0}h_{\mu}(t)$.  Moreover, $v_0(t,x)\ge v(t,x)$ and $h_*(t)\ge h_0$ for $t\ge0$ and $x\in[0,h_0]$.

 Obviously, for any $t>0$, we have
 \[h_0\le h_{\mu}(t)=h_0+\mu\dd\int_{0}^{t}\int_{0}^{h_{\mu}(\tau)}\int_{h_{\mu}(\tau)}^{\infty}
J(x,y)u_{\mu}(\tau,x)\dy\dx{\rm d}\tau\to h_0 ~ ~ {\rm as} ~ \mu\to0.\]
Hence $h_*(t)\equiv h_0$ for $t\ge0$. By Dini's theorem, $h_{\mu}(t)\to h_0$ locally uniformly in $\mathbb{R}^+$ as $\mu\to0$. Additionally, for any $t>0$ and $x\in[0,h_0]$, from the first equation of \eqref{1.2} we have
\[u_{\mu}=u_0(x)+d\dd\int_{0}^{t}\int_{0}^{h_{\mu}(\tau)}J(x,y)u_{\mu}(\tau,y)\dy{\rm d}\tau-dj(x)\dd\int_{0}^{t}u_{\mu}(\tau,x){\rm d}\tau+\dd\int_{0}^{t}f(u_{\mu}(\tau,x)){\rm d}\tau.\]
By the dominated convergence theorem, one easily gets
\[v_0=u_0(x)+d\dd\int_{0}^{t}\int_{0}^{h_0}J(x,y)v_0(\tau,y)\dy{\rm d}\tau-dj(x)\dd\int_{0}^{t}v_0(\tau,x){\rm d}\tau+\dd\int_{0}^{t}f(v_0(\tau,x)){\rm d}\tau,\]
which yields $v_0\equiv v$ for $t\ge0$ and $x\in[0,h_0]$. By Dini's theorem again, $u_{\mu}\to v$ locally uniformly in $\mathbb{R}^+\times[0,h_0]$. By the first and third equations in \eqref{1.2}, respectively, we easily see that $u_{\mu,t}\to v_t$ locally uniformly in $\mathbb{R}^+\times[0,h_0]$ and $h'_{\mu}\to0$ locally uniformly in $\mathbb{R}^+$ as $\mu\to0$.

\sk{\rm(2)}\, By some comparison arguments, we have that $(u_{\mu},h_{\mu})$ is increasing in $\mu>0$ and $u_{\mu}(t,x)\le V(t,x)$ for $t\ge0$ and $x\in[0,h_{\mu}(t)]$. So we can define $H(t)=\lim_{\mu\to\yy}h_{\mu}(t)$ and $u_\yy(t,x)=\lim_{\mu\to\yy}u_{\mu}(t,x)$ for $t\ge0$ and $x\in[0,H(t))$. We first show that $H(t)=\yy$ for any $t>0$. If there exists $t_0>0$ such that $H(t_0)<\yy$. Then $h_{\mu}(t)\le H(t)\le H(t_0)$ for $t\in[0,t_0]$. By the condition {\bf(J)}, there exist small $\sigma_0,\, \delta>0$ such that $J(x)\ge \sigma_0$ for $|x|\le2\delta$. Using the third equation of \eqref{1.2} and the dominated convergence theorem we have that, for $t\in(0,t_0]$,
\bess
h'_{\mu}(t)\ge\mu\dd\int_{h_{\mu}(t)-\delta}^{h_{\mu}(t)}\int_{h_{\mu}(t)}^{h_{\mu}(t)+\delta}
J(x,y)u_{\mu}(t,x)\dy\dx\ge\mu\sigma_0\delta\int_{h_{\mu}(t)-\delta}^{h_{\mu}(t)}u_{\mu}(t,x)\dx\to\yy ~ {\rm as } ~ \mu\to\yy,
\eess
which implies that $H(t_0)=\yy$. This contradiction shows that $H(t)=\yy$ for $t>0$. Furthermore, since $h_{\mu}(t)$ is increasing in $t>0$, we easily see that $\lim_{\mu\to\yy}h_{\mu}(t)=\yy$ locally uniformly in $\mathbb{R}^+$.

We now prove that $u_\yy$ satisfies \eqref{2.4}. For any $t>0$ and $x\in\overline{\mathbb{R}}^+$, we can choose $\mu$ large enough, say $\mu\ge\mu_1$, such that $x\in[0,h_{\mu}(t))$ for $\mu\ge\mu_1$. Clearly, there is $t_1\in(0,t)$ such that $x\in[0,h_{\mu}(t_1))$ for $\mu\ge\mu_1$. It then follows from \eqref{1.2} that, for $\mu\ge\mu_1$,
\[u_{\mu}=u_{\mu}(t_1,x)+d\dd\int_{t_1}^{t}\int_{0}^{h_{\mu}(\tau)}\!J(x,y)u_{\mu}(\tau,y)\dy{\rm d}\tau-dj(x)\dd\int_{t_1}^{t}u_{\mu}(\tau,x){\rm d}\tau+\dd\int_{t_1}^{t}f(u_{\mu}(\tau,x)){\rm d}\tau.\]
Using the dominated convergence theorem again, we derive
\[u_\yy=u_\yy(t_1,x)+d\dd\int_{t_1}^{t}\int_{0}^{\yy}J(x,y)u_\yy(\tau,y)\dy{\rm d}\tau-dj(x)\dd\int_{t_1}^{t}u_\yy(\tau,x){\rm d}\tau+\dd\int_{t_1}^{t}f(u_\yy(\tau,x)){\rm d}\tau.\]
By differentiating the above equation by $t$, we see
\[u_{\yy,t}=d\dd\int_{0}^{\yy}J(x,y)u_\yy(t,y)\dy-dj(x)u_\yy(t,x)+f(u_\yy(t,x)).\]
Moreover, it is easy to show that $u_\yy(0,x)=u_0(x)$ for $x\in[0,h_0]$ and $u_\yy(0,x)\equiv0$ for $x>h_0$. By the uniqueness of solutions, $u_\yy\equiv V$. It then follows from Dini's theorem that $u_{\mu}\to V$ locally uniformly in $\mathbb{R}^+\times\mathbb{R}^+$ as $\mu\to\yy$. The proof is complete.
 \end{proof}

\section{Sharp estimates for the spreading speed of \eqref{1.2}}\lbl{s3}
\setcounter{equation}{0} {\setlength\arraycolsep{2pt}

In this section, some sharp estimates for \eqref{1.2} will be established by following analogous lines in the proofs of \cite[Theorem 1.4]{DN21} and \cite[Theorem 1.6]{DN212}.

\begin{theorem}\label{t3.1} Let $J$ be compactly supported, $f\in C^2$ and satisfy the condition {\bf(F)}. Let $(u,h)$ be the solution of \eqref{1.2}. If spreading happens, then there exist $T$ and $C>0$ such that
\[|h(t)-c_0t|\le C\ln t ~ ~ {\rm for } ~ t\ge T.\]
\end{theorem}

This theorem will be proved by the following several lemmas.

\begin{lemma}\label{l3.1} Let the condition {\bf(J1)} hold, $f\in C^2$ and satisfy the condition {\bf(F)}. Let $(u,h)$ be the solution of \eqref{1.2} and spreading happen. Then there exists $C>0$ such that
\[h(t)-c_0t\le C ~ ~ {\rm for } ~ t\ge0.\]
\end{lemma}

\begin{proof}For the positive constants $\beta>1$ and $\theta,\,l\gg 1$, which will be determined later, we define
 \bess
&\ep(t)=(t+\theta)^{-\beta}, ~ ~ \delta(t)= l+\frac{c_0}{1-\beta}\big[(t+\theta)^{1-\beta}-\theta^{1-\beta}\big], ~ ~ \bar{h}(t)=c_0t+\delta(t),&\\[1mm]
 &\bar{u}(t,x)=(1+\ep(t))\phi^{c_0}(x-\bar{h}(t))\;\;\;\text{ for }\;t\ge0,\;\; x\in[0,\bar{h}(t)].&\eess
We will show that there exist adequate $\theta,\,  l$ and $T>0$ such that
\bess
\left\{\begin{aligned}
&\bar u_t\ge d\dd\int_{0}^{\bar h(t)}J(x,y)\bar u(t,y)\dy-dj(x)\bar u+ f(\bar u), && t>0,~0\le x<\bar h(t),\\
&\bar u(t,\bar h(t))\ge0,&& t>0,\\
&\bar h'(t)\ge\mu\dd\int_{0}^{\bar h(t)}\!\!\int_{\bar h(t)}^{\infty}
J(x,y)\bar u(t,x)\dy\dx,&& t>0,\\
&\bar h(0)\ge h(T),\;\; \bar u(0,x)\ge u(T,x),&& x\in[0,h(T)].
\end{aligned}\right.
 \eess
Once it is done, it directly follows from the comparison principle (\cite[Theorem 3.7]{LW21}) that
 \[u(t+T,x)\le\bar{u}(t,x), ~ ~ h(t+T)\le \bar{h}(t) ~ ~ ~ {\rm for } ~ t\ge0, ~ x\in[0,h(t+T)],\]
 which immediately yields the desired result.

 We first prove that if $\theta$ is large sufficiently, then
 \bes\label{3.1}
 \bar u_t\ge d\dd\int_{0}^{\bar h(t)}J(x,y)\bar u(t,y)\dy-dj(x)\bar u+ f(\bar u) ~ ~ {\rm for } ~ t>0 ~ x\in[0,\bar{h}(t)).
 \ees
 Direct computations show that
 \bess
 \bar{u}_t&=&-(1+\ep(t))(c_0+\delta'(t)){\phi^{c_0}}'(x-\bar{h}(t))+\ep'(t)\phi^{c_0}(x-\bar{h}(t))\\
 &=&(1+\ep(t))\left(d\dd\int_{-\yy}^{\bar h(t)}\!J(x,y)\phi^{c_0}(y-\bar{h}(t))\dy-d\phi^{c_0}(x-\bar{h}(t))+f(\phi^{c_0}(x-\bar{h}(t)))\right)\\
 &&-(1+\ep(t))\delta'(t){\phi^{c_0}}'(x-\bar{h}(t))+\ep'(t)\phi^{c_0}(x-\bar{h}(t))\\
 &\ge&d\dd\int_{0}^{\bar h(t)}J(x,y)\bar u(t,y)\dy-dj(x)\bar u+f(\bar u)+A(t,x),
 \eess
 where
 \bess
 A(t,x)&=&(1+\ep(t))f(\phi^{c_0}(x-\bar{h}(t)))-f((1+\ep(t))\phi^{c_0}(x-\bar{h}(t)))\\
 &&-(1+\ep(t))\delta'(t){\phi^{c_0}}'(x-\bar{h}(t))+\ep'(t)\phi^{c_0}(x-\bar{h}(t)).
 \eess
It suffices to show that $A(t,x)\ge0$ for $t\ge0$ and $x\in[0,\bar{h}(t)]$. Let $C=\max_{u\in[0,2u^*]}|f''(u)|$. By the Taylor expansion, we have
 \bess
 (1+\ep)f(u)-f((1+\ep)u)&=&-f((1+\ep)u^*)+(1+\ep)
 \left[f'(\tilde{u})-f'((1+\ep)\tilde{u})\right](u-u^*)\\
 &\ge&-\ep u^*f'(u^*)+o(\ep)-(1+\ep)C\ep u^*(u^*-u)\;\;\;\text{for}\;\;u\in[0,u^*].
 \eess
As $\phi^{c_0}(-\yy)=u^*$, for any small $\ep_0>0$ there exists $ l_1>0$ such that $\phi^{c_0}(- l_1)\ge (1-\ep_0)u^*$. So $\phi^{c_0}(x-\bar{h}(t))\in[(1-\ep_0)u^*,u^*]$ for $x\in[0,\bar{h}(t)- l_1]$. When $\theta\gg 1$, $l>l_1$ and $0<\ep_0\ll 1$, we have
 \bess
 A(t,x)&\ge&(1+\ep(t))f(\phi^{c_0}(x-\bar{h}(t)))-f((1+\ep(t))\phi^{c_0}(x-\bar{h}(t)))-\frac{\beta}{(t+\theta)^{\beta+1}}\phi^{c_0}(x-\bar{h}(t))\\
 &\ge&-\ep(t) u^*f'(u^*)+o(\ep(t))-(1+\ep(t))C\ep(t) u^*\ep_0u^*-\ep(t)\frac{\beta}{t+\theta}u^*\\
 &\ge&\ep(t)\left[-u^*f'(u^*)+o(1)-2Cu^*\ep_0u^*-{\beta}u^*/\theta\right]\ge0 \;\;
 \text{for}\;\; x\in[0,\bar{h}(t)- l_1],
 \eess
and
 \bess A(t,x)&\ge&-(1+\ep(t))\delta'(t){\phi^{c_0}}'(x-\bar{h}(t))
 -\frac{\beta}{(t+\theta)^{\beta+1}}\phi^{c_0}(x-\bar{h}(t))\\
 &\ge&c_0\ep_1\ep(t)-\frac{\beta}{(t+\theta)^{\beta+1}}u^*\\
 &\ge&(t+\theta)^{-\beta-1}(c_0\ep_1\theta-\beta u^*)\ge0\;\;
 \text{for}\;\;x\in[\bar{h}(t)- l_1,\bar{h}(t)],
 \eess
where $\ep_1=\inf_{x\in[- l_1,\,0]}(-{\phi^{c_0}}'(x))>0$. Hence \eqref{3.1} holds. Moreover, simple calculations yield
 \bess
 \mu\dd\int_{0}^{\bar h(t)}\!\!\int_{\bar h(t)}^{\infty}
J(x,y)\bar u(t,x)\dy\dx&=&\mu(1+\ep(t))\dd\int_{0}^{\bar h(t)}\!\!\int_{\bar h(t)}^{\infty}
J(x,y)\phi^{c_0}(x-\bar{h}(t))\dy\dx\\
&\le&\mu(1+\ep(t))\dd\int_{-\yy}^{0}\int_{0}^{\infty}
J(x,y)\phi^{c_0}(x)\dy\dx\\
&=&(1+\ep(t))c_0=\bar{h}'(t).
 \eess
Since $\limsup_{t\to\yy}u(t,x)\le u^*$ uniformly in $[0,h(t)]$, for $\theta$ chosen as above, there is $T>0$ such that $u(T,x)\le (1+\frac{\ep(0)}2)u^*$ for $x\in[0,h(T)]$. Together with $\phi^{c_0}(-\yy)=u^*$, one may choose $ l$ large sufficiently, if necessary, such that $\bar{h}(0)= l>h(T)$, and
 \[\bar{u}(0,x)=(1+\ep(0))\phi^{c_0}(x- l)\ge(1+{\ep(0)}/2)u^*\ge u(T,x) ~ ~ {\rm for } ~ x\in[0,h(T)].\]
The proof is complete.
\end{proof}

Now we prove a crucial estimate for the solution of \eqref{1.2}. Obviously, the condition {\bf(F)} implies that there is a positive constant $\rho$ depending only on $f$ such that
\bes\label{3.2}f(u)\ge \rho\min\{u,\, u^*-u\} ~ ~ ~ {\rm for} ~ u\in[0,u^*].
\ees

\begin{lemma}\label{l3.2} Let $[-r,r]$ be the smallest compact set which contains the support of $J$ and spreading happen for the problem \eqref{1.2}. Let $(u,h)$ be the solution of \eqref{1.2}. Then there exist positive constants $\eta_1,\,\eta_2$ and $\theta_1$ such that for any $\theta\ge\theta_1$, one can find a $T>0$ depending on $\eta_1$ and $\theta$ such that
\[u(t+T,x)\ge u^*-{\eta_2}/(t+\theta) ~ ~ {\rm for } ~ t\ge0, ~ x\in[0,\,\eta_1(t+\theta)].\]
\end{lemma}

\begin{proof} Take $\theta\gg 1$ to be determined later and
  \[0<\eta_1<\min\kk\{\frac{c_0}2,\,\frac{\rho}{8},\,\frac{\rho r}{12},\,\frac{dr}{36}\int_{\frac{2r}{3}}^rJ(y)\dy\rr\}, \;\;\;0<\rho_1<\eta_1\theta u^*.\]
We define
\[\underline{h}(t)=2\eta_1(t+\theta), ~ ~ \underline{u}(t,x)=\left\{\begin{aligned}
&u^*-{\rho_1}/{\underline{h}(t)},& &t\ge0, ~ x\in\kk[0,\,{\underline{h}(t)}/2\rr],\\[2mm]
&2\left[u^*-{\rho_1}/{\underline{h}(t)}\right]\left[1-{x}/{\underline{h}(t)}\right],& &t\ge0, ~ x\in\kk[{\underline{h}(t)}/2,\,\underline{h}(t)\rr].
\end{aligned}\right.\]
Obviously, $\underline{u}$ is continuous and nonnegative for $t\ge0$ and $x\in[0,\underline{h}(t)]$, and $\underline{u}_t$ is continuous for $t>0$ and $x\in[0,\underline{h}(t)]\setminus \{\underline{h}(t)/2\}$.
We will prove that there exist suitable $\theta$ and $T$ such that
\bes\label{3.3}
\left\{\begin{aligned}
&\underline u_t\le d\dd\int_{0}^{\underline h(t)}J(x,y)\underline u(t,y)\dy-dj(x)\underline u+ f(\underline u), && t>0,~x\in[0,\underline{h}(t)]\setminus\kk\{{\underline{h}(t)}/2\rr\},\\
&\underline u(t,\underline h(t))=0,&& t>0,\\
&\underline h(t)\le h(t+T),\;t\ge0; ~ ~ ~ ~ ~ \underline u(0,x)\le u(T,x),&& x\in[0,\underline h(0)].
\end{aligned}\right.
 \ees
If \eqref{3.3} is satisfied, we can compare $\underline{u}(t,x)$ and $u(t+T,x)$ over $\{(t,x):t\ge0, ~ x\in[0,\underline{h}(t)]\}$ to derive $u(t+T,x)\ge\underline{u}(t,x)$ for $t\ge0$ and $x\in[0,\underline{h}(t)]$ which implies our assertion.

We first verify the first inequality of \eqref{3.3}. Take $\theta\gg 1$ such that $\eta_1\theta\gg r$. We will discuss it in four cases.
To save space, in the following we set
 \[p(t)=u^*-\frac{\rho_1}{\underline{h}(t)},\;\;\; q(x,t)=1-\frac{x}{\underline{h}(t)}.\]

{\bf  Case 1}: $x\in[0,\,{\underline{h}(t)}/2-r]$. It follows from \eqref{3.2} that
  \bess d\int_{0}^{\underline h(t)}\!\!J(x,y)\underline u(t,y)\dy-dj(x)\underline u+ f(\underline u)=d\int_{0}^{\frac{\underline h(t)}2}\!\!J(x,y)\left(\underline u(t,y)-\underline{u}\right)\dy+f(\underline{u})
  =f(\underline{u})\ge \frac{\rho\rho_1}{\underline{h}(t)}.
  \eess
  On the other hand, for $x\in[0,\,{\underline{h}(t)}/2-r]$, we have
\[\underline{u}_t(t,x)=\frac{\rho_1\underline{h}'(t)}{\underline{h}^2(t)}=\frac{\rho_1}{2\eta_1(t+\theta)^2}\le\frac{\rho\rho_1}{\underline{h}(t)}\]
provide that $\rho\theta\ge1$. The first inequality of \eqref{3.3} holds in Case 1.

{\bf Case 2}: $x\in[{\underline{h}(t)}/2-r,\,{\underline{h}(t)}/2+r]\setminus \{{\underline{h}(t)}/2\}$.
When $x\in[{\underline{h}(t)}/{2}-r,\,{\underline{h}(t)}/{2}]$, we have
 \bess
 \int_{0}^{\underline h(t)}\!\!J(x,y)\underline u(t,y)\dy-j(x)\underline u
 &\ge&\int_{-r}^r\!J(y)\underline{u}(t,x+y)\dy-p(t)\\
 &=&\left\{\!\int_{-r}^{\frac{\underline{h}(t)}2-x}
 +\int_{\frac{\underline{h}(t)}2-x}^r\!\right\}\!J(y)\underline{u}(t,x+y)\dy
 -p(t)\!\int_{-r}^r\!J(y)\dy\quad\;\\
 &=&\int_{\frac{\underline{h}(t)}2-x}^rJ(y)\underline{u}(t,x+y)\dy
 -p(t)\!\int_{\frac{\underline{h}(t)}2-x}^rJ(y)\dy\\
 &=&p(t)\int_{\frac{\underline{h}(t)}2-x}^rJ(y)
 \left[2q(x+y,t)-1\right]\dy\\
 &\ge& p(t)\int_{\frac{\underline{h}(t)}2-x}^rJ(y)
 \left[2q(r+\underline{h}(t)/2,\,t)-1\right]\dy\ge\frac{-2ru^*}{\underline{h}(t)}.
 \eess
When $x\in{\ccc(}{\underline{h}(t)}/2,{\underline{h}(t)}/2+r\big]$, we derive
 \bess
 \underline{u}(t,x)=2p(t)q(x,t)\int_{-r}^rJ(y)\dy-2p(t)
 \int_{-r}^rJ(y)\frac{y}{\underline{h}(t)}\dy
 =2p(t)\int_{-r}^rJ(y)q(x+y,\,t)\dy,
 \eess
and furthermore,
 \bess
 \int_{0}^{\underline h(t)}\!\!J(x,y)\underline u(t,y)\dy-j(x)\underline u
 &\ge&\left\{\!\int_{-r}^{\frac{\underline{h}(t)}2-x}
 +\int_{\frac{\underline{h}(t)}2-x}^r\!\right\}
 J(y)\underline{u}(t,x+y)\dy-2p(t)q(x,t)\!\int_{-r}^r\!J(y)\quad\;\\
 &=&p(t)\int_{-r}^{\frac{\underline{h}(t)}2-x}J(y)\left[1
 -2q(x+y,t)\right]\dy
 \ge \frac{-2ru^*}{\underline{h}(t)}.
 \eess
It is easy to see that, when $\theta\ge2r/\eta_1$,
\bess
\min\left\{\underline u,\,u^*-\underline u\right\}&\ge& \min\left\{2p(t)q({r+\underline{h}(t)}/2,\,t),\;\frac{\rho_1}{\underline{h}(t)}\right\}\\
 &\ge&\frac{\rho_1}{\underline{h}(t)}\min\left\{q(2r,\,t),\;1\right\}\\
 &\ge&\frac{\rho_1}{\underline{h}(t)}\min\left\{1-\frac{r}{\eta_1\theta},\;1\right\}\\
 &\ge&\frac{\rho_1}{2\underline{h}(t)} ~ ~ ~ {\rm for} ~ x\in[{\underline{h}(t)}/2-r,\,{\underline{h}(t)}/2+r]. \eess
Combining these estimates with \eqref{3.2} we have that, for $x\in[{\underline{h}(t)}/2-r,\,{\underline{h}(t)}/2+r]$,
 \bess
 d\int_{0}^{\underline h(t)}J(x,y)\underline u(t,y)\dy-dj(x)\underline u(t,x)+ f(\underline u)\ge \frac{-2rdu^*}{\underline{h}(t)}+f(\underline{u})\ge \frac{\rho\rho_1}{2\underline{h}(t)}-\frac{2rdu^*}{\underline{h}(t)}.
 \eess
Additionally, we easily get $\underline{u}_t=\frac{2\rho_1\eta_1}{\underline{h}^2(t)}$ for $x\in[\frac{\underline{h}(t)}2-r,\, \frac{\underline{h}(t)}2)$,
 and
 \[\underline{u}_t=\frac{2\rho_1\underline{h}'(t)}{\underline{h}^2(t)}
 q(x,t)+2p(t)\frac{x\underline{h}'(t)}{\underline{h}^2(t)}
 \le\frac{4\rho_1\eta_1}{\underline{h}^2(t)}+\frac{4u^*\eta_1}{\underline{h}(t)}
 \;\;\;\text{for}\;\; x\in\kk(\frac{\underline{h}(t)}2,\frac{\underline{h}(t)}2+r\rr].\]
Hence, the first inequality of \eqref{3.3} holds for Case 2 provided that $\rho\rho_1\ge4rdu^*+8\rho_1\eta_1+8u^*\eta_1$,
which can be guaranteed by choosing $\rho_1,\,\theta$ large enough.

{\bf Case 3}: $x\in[{\underline{h}(t)}/2+r,\, \underline{h}(t)-r]$. Direct computations yield
\bess
\int_{0}^{\underline h(t)}\!\!J(x,y)\underline u(t,y)\dy-\underline u
&=&\int_{-r}^r\!\!J(y)\underline u(t,x+y)\dy-\underline u\\
&=&2\int_{-r}^r\!\!J(y)p(t)
q(x+y,t)\dy-2p(t)
q(x,t)\\
&=&2\int_{-r}^rJ(y)p(t)\frac{-y}{\underline{h}(t)}\dy=0,
\eess
and
\[\min\{\underline u,\,u^*-\underline u\}\ge\min\kk\{2p(t)
\frac{r}{\underline{h}(t)},\; \frac{\rho_1}{\underline{h}(t)}\rr\}
\ge\min\kk\{\frac{ru^*}{\underline{h}(t)},\,\frac{\rho_1}{\underline{h}(t)}\rr\}
\ge\frac{ru^*}{\underline{h}(t)}\]
with $\rho_1\ge ru^*$.
Moreover,
\bess\underline{u}_t=\frac{2\rho_1\underline{h}'(t)}{\underline{h}^2(t)}
q(x,t)+2p(t)
\frac{x\underline{h}'(t)}{\underline{h}^2(t)}
\le\frac{4\rho_1\eta_1}{\underline{h}^2(t)}+\frac{4u^*\eta_1}{\underline{h}(t)}\le\frac{6u^*\eta_1}{\underline{h}(t)}.
\eess
Then from the choice of $\eta_1$ we see that the first inequality of \eqref{3.3} holds true in Case 3.

{\bf Case 4}: $x\in[\underline{h}(t)-r,\, \underline{h}(t)]$. In this case we have
\bess
 \int_{0}^{\underline h(t)}\!\!J(x,y)\underline u(t,y)\dy-\underline u
 &=&\int_{-r}^r\!J(y)\underline u(t,x+y)\dy-\underline u-2p(t)\!\int_{\underline h(t)-x}^r\!\!J(y)q(x+y,t)\dy\\
&=&2p(t)\int_{-r}^rJ(y)\frac{{-y}}{\underline{h}(t)}\dy-2p(t)\int_{\underline h(t)-x}^rJ(y)q(x+y,t)\dy\\
&\ge&-u^*\int_{\underline h(t)-x}^rJ(y)q(x+y,t)\dy\ge0,
\eess
and
\[\min\{\underline u,\,u^*-\underline u\}\ge \min\kk\{u^*q(x,t),\, \frac{\rho_1}{\underline{h}(t)}\rr\}\ge u^*q(x,t)\;\;\;\text{if}\;\;\rho_1\ge ru^*.\]
Therefore, we obtain
 \bess
-du^*\int_{\underline h(t)-x}^rJ(y)q(x+y,t)\dy+f(\underline{u})&\ge& \rho u^*q(x,t)\ge \frac{\rho u^*r}{2\underline{h}(t)}\;\;\;\text{when}\;\;x\in[\underline{h}(t)-r,\, \underline{h}(t)-{r}/2],
\eess
and
\bess
-du^*\int_{\underline h(t)-x}^rJ(y)q(x+y,t)\dy+f(\underline{u})&\ge&-du^*\int_{\underline h(t)-x}^rJ(y)q(x+y,t)\dy\\
&\ge&-du^*\int_{\frac{2r}{3}}^rJ(y)q(x+y,t)
\dy\\
&\ge&\frac{dru^*}{6\underline{h}(t)}\int_{\frac{2r}{3}}^rJ(y)\dy\;\;\;\text{when}\;\;x\in[\underline{h}(t)-{r}/2,\, \underline{h}(t)].
 \eess
It then follows that, for $x\in[\underline{h}(t)-r,\, \underline{h}(t)]$,
\bess
d\int_{0}^{\underline h(t)}J(x,y)\underline u(t,y)\dy-dj(x)\underline u(t,x)+ f(\underline u)\ge \frac{\rho_2}{\underline{h}(t)}
\eess
with $\rho_2:=\min\{\frac{\rho u^*r}2,\,\frac{dru^*}{6}\int_{\frac{2r}{3}}^rJ(y)\dy\}$.
Similarly to Case 3, we have $\underline{u}_t\le\frac{6u^*\eta_1}{\underline{h}(t)}$. Thus  the first inequality of \eqref{3.3} holds in this case with adequate choice of $\eta_1$.

Due to the above analysis, one can see that the first inequality of \eqref{3.3} holds if $\theta$ is large suitably, say $\theta\ge\theta_1$ which depends  only on initial data.

For any $\theta\ge\theta_1$ and $\eta_1$ chosen as above, we next show that there is $T>0$ such that the last two inequalities in \eqref{3.3} hold. Clearly, $\underline{u}(0,x)\le u^*-\frac{\rho_1}{2\eta_1\theta}$. Since spreading happens for \eqref{1.2}, $J$ is compactly supported and $2\eta_1<c_0$, there is $T>0$ depending on $\eta_1$ and $\theta$ such that $h(t+T)\ge2\eta_1(t+\theta)=\underline{h}(t)$ for $t\ge0$ and $u(T,x)\ge\underline{u}(0,x)$ for $x\in[0,\underline{h}(0)]$. Then we see that $u(t+T,\underline{h}(t))>0=\underline{u}(t,\underline{h}(t))$ for $t>0$.
So the proof is finished.
\end{proof}

\begin{lemma}\label{l3.3}Under the assumptions of Lemma \ref{l3.2}, there exist $C>0$ and $T\gg1$ such that
\[h(t)-c_0t\ge-C\ln t ~ ~ {\rm for } ~ t\ge T.\]
\end{lemma}

\begin{proof} Let $\theta\ge\theta_1$, $ l_1\ge \eta_2/u^*$ and $ l_2>0$. We define
 \bess
 \delta(t)&=&c_0\theta- l_2\left[\ln(t+\theta)-\ln \theta\right],\;\;\;\ep(t)= l_1/(t+\theta),\\
 \underline{h}(t)&=&c_0t+\delta(t), \;\; \underline{u}(t,x)=(1-\ep(t))\phi^{c_0}(x-\underline{h}(t)).
  \eess
Clearly, $\underline u(t,\underline h(t))=(1-\ep(t))\phi^{c_0}(0)=0$, and when $\theta\gg1$,
  \[c_0(t+\theta)\ge\underline{h}(t)\ge c_0(t+\theta)/2 ~ ~ ~ {\rm for } ~ t\ge0.\]
We will show that there exist suitable $\theta,\, l_1,\, l_2$ and $T>0$  such that
  \bes\label{3.4}
\left\{\begin{aligned}
&\underline u_t\le d\dd\int_{0}^{\underline h(t)}J(x,y)\underline u(t,y)\dy-dj(x)\underline u+ f(\underline u), && t>0,~\eta_0\underline{h}(t)<x<\underline h(t),\\
&\underline h'(t)\le\mu\dd\int_{0}^{\underline h(t)}\int_{\underline h(t)}^{\infty}
J(x,y)\underline u(t,x)\dy\dx,&& t>0,\\
&\underline u(t,x)\le u(t+T,x),&& t>0, ~0\le x\le \eta_0\underline{h}(t),\\
&\underline h(0)\le h(T),\;\;\underline u(0,x)\le u(T,x),&& x\in[0,\underline h(0)],
\end{aligned}\right.
\ees
where $\eta_0=\eta_1/c_0,\,\eta_1,\,\eta_2$ and $\theta_1$ are determined by Lemma \ref{l3.2}. Once \eqref{3.4} is obtained, the desired assertion can be deduced by the comparison principle (\cite[Theorem 3.8]{LW21}).

Obviously, taking $ l_2\ge  l_1c_0$ and $\theta\gg1$ we have
\bess
\mu\dd\int_{0}^{\underline h(t)}\!\int_{\underline h(t)}^{\infty}\!
J(x,y)\underline u(t,x)\dy\dx&=&\mu(1-\ep(t))\int_{0}^{\underline h(t)}\!\int_{\underline h(t)}^{\infty}\!J(x,y)\phi^{c_0}(x-\underline{h}(t))\dy\dx\\
&=&\mu(1-\ep(t))\int_{-\underline h(t)}^{0}\!\int_{0}^{\infty}\!
J(x,y)\phi^{c_0}(x)\dy\dx\\
&=&(1-\ep(t))c_0\ge c_0-\frac{ l_2}{t+\theta}=\underline{h}'(t).
\eess
Thus the second inequality in \eqref{3.4} holds.  Direct calculations show that, for $\theta\gg1$,
\bess
 \underline{u}_t&=&-\ep'(t)\phi^{c_0}(x-\underline{h}(t))-(1-\ep(t))(c_0+\delta'(t)){\phi^{c_0}}'(x-\underline{h}(t))\\
&=&(1-\ep(t))\left(d\dd\int_{-\yy}^{\underline h(t)}J(x,y)\phi^{c_0}(y-\underline{h}(t))\dy-d\phi^{c_0}(x-\underline{h}(t))+f(\phi^{c_0}(x-\underline{h}(t)))\right)\\
&& ~ ~ ~ -\ep'(t)\phi^{c_0}(x-\underline{h}(t))-(1-\ep(t))\delta'(t){\phi^{c_0}}'(x-\underline{h}(t))\\
&\le& d\dd\int_{0}^{\underline h(t)}J(x,y)\underline u(t,y)\dy-dj(x)\underline u+ f(\underline u)+B(t,x)
\eess
with
\[B(t,x)=(1-\ep(t))f(\phi^{c_0}(x-\underline{h}(t)))-f(\underline u)-\ep'(t)\phi^{c_0}(x-\underline{h}(t))-(1-\ep(t))\delta'(t){\phi^{c_0}}'(x-\underline{h}(t)).\]
Thus, to prove the first inequality in \eqref{3.4}, it remains to check $B(t,x)\le0$ for $t>0$ and $x\in(\eta_0\underline{h}(t),\underline h(t))$. Since $\phi^{c_0}(-\yy)=u^*$, for any small $\ep_0>0$ there is $\kappa>0$ such that $\phi^{c_0}(-\kappa)\ge(1-\ep_0)u^*$. Similarly to Lemma \ref{l3.1}, by the Taylor expansion, one has
\bess
  (1-\ep)f(u)-f((1-\ep)u)&=&-f((1-\ep)u^*)+(1-\ep)\left[f'(\tilde{u})-f'((1-\ep)\tilde{u})\right](u-u^*)\\
&=&f'(u^*)\ep u^*+o(\ep)+(1-\ep)\left[f'(\tilde{u})-f'((1-\ep)\tilde{u})\right](u-u^*)\\
&\le& f'(u^*)\ep u^*+o(\ep)-(1-\ep)C_1\ep u^*(u-u^*),
\eess
where $C_1$ depends only on $f$. For $\eta_0\underline{h}(t)<x\le\underline{h}(t)-\kappa$, we have $\phi^{c_0}(x-\underline{h}(t))-u^*\ge-\ep_0u^*$. So
\bess
(1-\ep(t))f(\phi^{c_0}(x-\underline{h}(t)))-f(\underline u)&\le& f'(u^*)\ep u^*+o(\ep)-(1-\ep)C_1\ep u^*(\phi^{c_0}(x-\underline{h}(t))-u^*)\\
&\le&  f'(u^*)\ep u^*+o(\ep)+(1-\ep)C_1\ep {u^*}^2\ep_0\le-C_2\ep,
\eess
where $C_2$ depends only on $f$, $0<\ep_0\ll 1$ and $\theta\gg 1$. It follows that
\bess
B(t,x)&\le& (1-\ep(t))f(\phi^{c_0}(x-\underline{h}(t)))-f(\underline u)-\ep'(t)\phi^{c_0}(x-\underline{h}(t))\\
&\le&-C_2\ep+\frac{ l_1}{(t+\theta)^2}u^*\\
&=&\ep\left(-C_2+\frac{u^*}{t+\theta}\right)\le0 \;\;\;\text{for}\;\; \eta_0\underline{h}(t)<x\le\underline{h}(t)-\kappa,
\eess
and
\bess
B(t,x)&\le&-\ep'(t)\phi^{c_0}(x-\underline{h}(t))-(1-\ep(t))\delta'(t){\phi^{c_0}}'(x-\underline{h}(t))\\
&\le&-\frac{(1-\ep)C_3 l_2}{t+\theta}+\frac{ l_1u^*}{(t+\theta)^2}\\
&\le&-\frac{C_3 l_2}{2(t+\theta)}+\frac{ l_1u^*}{(t+\theta)^2}\\
&=&\frac{1}{t+\theta}\left(\frac{-C_3 l_2}2+\frac{ l_1u^*}{t+\theta}\right)\le0\;\;\;\text{for}\;\; \underline{h}(t)-\kappa\leq x<\underline{h}(t),
\eess
where $C_3=\inf_{x\in[-\kappa,0]}(-{\phi^{c_0}}'(x))>0$.

We now show the last two inequalities in \eqref{3.4}. For these $ l_1,\, l_2$ and $\theta$ taken as above, thanks to $ l_1\ge \eta_2/u^*$ and $\eta_0\underline{h}(t)\le \eta_1(t+\theta)$, it follows from Lemma \ref{l3.2} that
\[\underline{u}(t,x)\le \kk(1-{ l_1}/(t+\theta)\rr)u^*\le u^*-{\eta_2}/(t+\theta)\le u(t+T,x) ~ ~ {\rm for } ~ t>0, ~ x\in[0,\eta_0\underline{h}(t)].\]
Moreover, since spreading happens, one can choose $T$ large enough such that $\underline{h}(0)\le h(T)$ and $\underline{u}(0,x)\le (1-\ep(0))u^*\le u(T,x)$ for $x\in[0,\underline{h}(0)]$. Thus \eqref{3.4} holds, and the proof is complete.
\end{proof}
Clearly, Theorem \ref{t3.1} follows from Lemmas \ref{l3.1} and \ref{l3.3}.

\section{Rate of accelerated spreading}
\setcounter{equation}{0} {\setlength\arraycolsep{2pt}

In this section, we assume that $J$ satisfies ${\bf(J^\gamma)}$.
It is easy to show that the condition {\bf(J)} holds if and only if $\gamma>1$, and the condition {\bf(J1)} holds if and only if $\gamma>2$. Now we focus on the case with $\gamma\in(1,2]$ which implies that accelerated spreading can happen for the problem  \eqref{1.2}. Enlightened by \cite{DN212}, we have the following theorem which will be proved by several lemmas.

\begin{theorem}\label{t4.1} Let the condition ${\bf(J^\gamma)}$ hold with $\gamma\in(1,2]$, $f$ satisfy the condition {\bf(F)} and spreading happen for \eqref{1.2}. Then, when $t\gg1$,
\bess\left\{\begin{aligned}
&h(t)\approx t^{\frac{1}{\gamma-1}} ~ ~ {\rm and} ~ \lim_{t\to\yy}\max_{[0,\,s(t)]}|u(t,x)- u^*|=0 ~ {\rm for ~ any } ~ 0\le s(t)=t^{\frac{1}{\gamma-1}}o(1) ~ ~ {\rm  ~ if ~}\gamma\in(1,2),\\
&h(t)\approx t\ln t ~ ~ {\rm and} ~ \lim_{t\to\yy}\max_{[0,\,s(t)]}|u(t,x)- u^*|=0 ~ {\rm for ~ any } ~ 0\le s(t)=(t\ln t) o(1)  ~ ~ {\rm  ~ if ~ }\gamma=2.
\end{aligned}\right.\eess
\end{theorem}

\begin{lemma}\label{l4.1} Under the assumptions of Theorem \ref{t4.1}, there is $C>0$ such that, when $t\gg1$,
 \bess
h(t)\le Ct^{\frac{1}{\gamma-1}} ~ ~ {\rm if ~}\gamma\in(1,2), ~ ~ {\rm and} ~ ~
h(t)\le Ct\ln t ~ ~ {\rm if ~ }\gamma=2.
\eess
\end{lemma}

\begin{proof}Clearly, we have
 \bess
 \int_{0}^{h}\!\int_{h}^{\infty}\!J(x,y)\dy\dx&=&\int_{0}^{1}J(y)y\dy+\int_{1}^{h}J(y)y\dy+h\int_{h}^{\yy}J(y)\dy\approx h^{2-\gamma}\;\;\;\text{if}\;\;\gamma\in(1,2),
 \eess
and
\bess
\int_{0}^{h}\!\int_{h}^{\infty}\!J(x,y)\dy\dx&=&\int_{0}^{1}J(y)y\dy+\int_{1}^{h}J(y)y\dy+h\int_{h}^{\yy}J(y)\dy\approx\ln h\;\;\;\text{if}\;\;\gamma=2.
 \eess
Moreover, there is $T>0$ such that $u(t,x)\le2u^*$ for $t\ge T$ and $x\in[0,h(t)]$. Hence, for $t\ge T$,
\[h'(t)\le2\mu u^*\int_{0}^{h(t)}\!\int_{h(t)}^{\infty}\!J(x,y)\dy\dx,\]
which implies our desired results.
\end{proof}

In order to give the lower estimate of $h(t)$ we should construct some suitable lower solution. To this aim, we first state a proposition which can be proved by similar arguments with \cite[Lemma 6.5]{DN21}.

\begin{proposition}\label{p4.1} Suppose that $\kappa_2>\kappa_1>0$ and $P(x)$ satisfies the condition {\bf(J)}. Define
\[\psi(x)=\min\left\{1,\,\frac{\kappa_2-|x|}{\kappa_1}\right\}.\]
Then for any small $\ep>0$, there is a $\kappa_{\ep}>0$ relying only on $P$ and $\ep$ such that if $\min\left\{\kappa_1,\,\kappa_2-\kappa_1\right\}\ge \kappa_{\ep}$ we have
\[\int_{0}^{\kappa_2}P(x-y)\psi(y)\dy\ge(1-\ep)\psi(x) ~ ~ {\rm for} ~ x\in[\kappa_{\ep},\kappa_2].\]
\end{proposition}

\begin{lemma}\label{l4.2}Suppose that the condition ${\bf(J^\gamma)}$ holds with $\gamma\in(1,2)$ and spreading happens for the problem \eqref{1.2}. Then there is $C>0$ such that
\bes\label{4.1}
 &h(t)\ge C t^{\frac{1}{\gamma-1}} ~ ~ ~ {\rm for} ~ t\gg1,&\\[1mm]
\label{4.2}
 &\dd\liminf_{t\to\yy}u(t,x)\ge u^* ~ ~ {\rm uniformly ~ in} ~ [0,s(t)] ~ ~ {\rm for ~ any } ~ 0\le s(t)=t^{\frac{1}{\gamma-1}}o(1).&
\ees
\end{lemma}

\begin{proof}For positive constants $\theta,\,l_1$ and $l_{\ep}=u^*-\sqrt{\ep}$ with $0<\ep\ll1$, we define
  \[\underline{h}(t)=( l_1t+\theta)^{\frac{1}{\gamma-1}}, ~ ~ \underline{u}(t,x)= l_{\ep}\min\left\{1,\,2\frac{\underline{h}(t)-x}{\underline{h}(t)}\right\} ~ ~ {\rm for } ~ t\ge0, ~ x\in[0,\underline{h}(t)].\]
We will show that there exist suitable $\theta,\, l_1$ and $T>0$  such that
  \bes\label{4.3}
\left\{\begin{aligned}
&\underline u_t\le d\dd\int_{0}^{\underline h(t)}J(x,y)\underline u(t,y)\dy-dj(x)\underline u+ f(\underline u), && t>0,~x\in[0,\underline h(t))\setminus\left\{\frac{\underline{h}(t)}2\right\},\\
&\underline u(t,\underline h(t))\le0,&& t>0,\\
&\underline h'(t)\le\mu\dd\int_{0}^{\underline h(t)}\int_{\underline h(t)}^{\infty}
J(x,y)\underline u(t,x)\dy\dx,&& t>0,\\
&\underline h(0)\le h(T),\;\;\underline u(0,x)\le u(T,x),&& x\in[0,\underline h(0)].
\end{aligned}\right.
\ees
If \eqref{4.3} holds, by the comparison principle (\cite[Theorem 3.7]{LW21}, which still holds for such situation; one can see \cite[Remark 2.4]{DN21} for an explanation), we immediately get
\[u(t+T,x)\ge \underline u(t,x) ~ ~ {\rm and} ~ ~ h(t+T)\ge\underline h(t), ~ ~ ~ {\rm for } ~ t\ge0 ~ x\in[0,\underline h(t)],\]
which indicates \eqref{4.1}. As for \eqref{4.2}, one easily deduces
\bess
\max_{[0,\,s(t)]}|\underline u(t,x)- u^*+\sqrt{\ep}|=(u^*-\sqrt{\ep})\left(1-\min\left\{1,\,2\frac{\underline{h}(t)-s(t)}{\underline{h}(t)}\right\}\right)\to0 ~ ~ {\rm as} ~ t\to\yy,
\eess
which, combining with \eqref{4.1}, yields $\liminf_{t\to\yy}u(t,x)\ge u^*-\sqrt{\ep}$ uniformly in $[0,s(t)]$.
By the arbitrariness of $\ep$, we derive \eqref{4.2}.

To this end, we first check the third inequality of \eqref{4.3}. Simple calculations show that for $\theta\gg1$,
  \bess
\mu\dd\int_{0}^{\underline h(t)}\!\int_{\underline h(t)}^{\infty}\!
J(x,y)\underline u(t,x)\dy\dx&\ge&2\mu  l_{\ep}\dd\int_{\frac{\underline{h}(t)}2}^{\underline h(t)}\!\int_{\underline h(t)}^{\infty}\!
J(x,y)\frac{\underline{h}(t)-x}{\underline{h}(t)}\dy\dx\\
&=&\frac{2\mu l_{\ep}}{\underline{h}(t)}\dd\int_{-\frac{\underline{h}(t)}2}^{0}\!\int_{0}^{\infty}\!J(x,y)(-x)\dy\dx\\
&=&\frac{2\mu  l_{\ep}}{\underline{h}(t)}\left(\int_{0}^{\frac{\underline{h}(t)}2}
\int_{0}^{y}\!\!J(y)x\dx\dy+\int_{\frac{\underline{h}(t)}2}^{\yy}\!\!
\int_{0}^{\frac{\underline{h}(t)}2}\!J(y)x\dx\dy\right)\\
 &\ge& \frac{2\mu l_{\ep}}{\underline{h}(t)}\int_{0}^{\frac{\underline{h}(t)}2}\int_{0}^{y}\!\!J(y)x\dx\dy\\
 &\ge&\frac{\mu l_{\ep}}{\underline{h}(t)}\int_{\frac{\underline{h}(t)}{4}}^{\frac{\underline{h}(t)}2}
J(y)y^2\dy\\
&\ge&\frac{\varsigma_1\mu l_{\ep}}{\underline{h}(t)}\int_{\frac{\underline{h}(t)}{4}}^{\frac{\underline{h}(t)}2}
y^{2-\gamma}\dy\ge\tilde{C} l_{\ep}\mu \underline{h}^{2-\gamma}(t)
\eess
with $\tilde{C}$ depending only on $J$. On the other hand,
\[\underline{h}'(t)=\frac{ l_1}{\gamma-1}( l_1t+\theta)^{\frac{2-\gamma}{\gamma-1}}=\frac{ l_1}{\gamma-1}\underline{h}^{2-\gamma}(t)\le\tilde{C} l_{\ep}\mu \underline{h}^{2-\gamma}(t)\]
provided that $\tilde{C} l_{\ep}\mu\ge\frac{ l_1}{\gamma-1}$. Thus, the third inequality in \eqref{4.3} holds.

Now we verify the first inequality of \eqref{4.3}. Clearly, $\underline{u}(t,x)\ge l_{\ep}q(x,t)$. For $x\in[{\underline{h}(t)}/4,\underline{h}(t)]$,
  \bess
\int_{0}^{\underline h(t)}J(x,y)\underline u(t,y)\dy&=&\int_{-x}^{\underline h(t)-x}J(y)\underline u(t,x+y)\dy\\
  &\ge&\int_{-\frac{\underline{h}(t)}{4}}^{-\frac{\underline{h}(t)}{8}}J(y)\underline u(t,x+y)\dy\\
&=& l_{\ep}\int_{-\frac{\underline{h}(t)}{4}}^{-\frac{\underline{h}(t)}{8}}J(y)
q(x+y,t)\dy\\
 &\ge& l_{\ep}\int_{-\frac{\underline{h}(t)}{4}}^{-\frac{\underline{h}(t)}{8}}
 \frac{\varsigma_1}{|y|^{\gamma}}q(x+y,t)\dy\\
&\ge& \frac{ l_{\ep}}{\underline{h}(t)}\int_{-\frac{\underline{h}(t)}{4}}^{-\frac{\underline{h}(t)}{8}}
\frac{\varsigma_1}{|y|^{\gamma}}(-y)\dy\ge\hat C l_{\ep}\underline{h}^{1-\gamma}(t)
\eess
with $\hat{C}$ depending only on $J$. For $x\in [0,{\underline{h}(t)}/{4}]$, we have
\bess
\int_{0}^{\underline h(t)}J(x,y)\underline u(t,y)\dy&\ge& l_{\ep}\int_{0}^{\frac{\underline{h}(t)}2}J(x,y)\dy\\
&=& l_{\ep}(j(x)-\ep)+ l_{\ep}\left(\ep-\int_{\frac{\underline{h}(t)}2-x}^{\yy}J(y)\dy\right)\\
 &\ge& l_{\ep}(j(x)-\ep)+ l_{\ep}\left(\ep-\int_{\frac{\underline{h}(t)}{4}}^{\yy}J(y)\dy\right)\\
&\ge&   l_{\ep}(j(x)-\ep)=(j(x)-\ep)\underline{u}(t,x)
\eess
provided that $\theta\gg1$. Moreover, from \eqref{3.2} we have that, for $x\in [0,{\underline{h}(t)}/{4}]$,
 \bes\label{4.4}
f(\underline u)\ge \rho\min\{\underline u,\,u^*-\underline u\}\ge \rho\min\{ l_{\ep},\,u^*-\underline u\}\ge \rho\sqrt{\ep}.
 \ees
Take $\kappa_2=\underline{h}(t)$ and $\kappa_1=\underline{h}(t)/2$ in Proposition \ref{p4.1} to obtain
\[\int_{0}^{\underline h(t)}J(x,y)\underline u(t,y)\dy\ge(1-\ep^2)\underline{u}(t,x) ~ ~ ~ {\rm for} ~ t>0, ~ x\in\kk[\frac{\underline{h}(t)}{4},\,\underline{h}(t)\rr].\]
Additionally, it follows from \eqref{3.2} that
\bes\label{4.5}
f(\underline u)\ge \rho\min\{\underline u,\,u^*-\underline u\}\ge \rho\ep\underline{u} ~ ~ {\rm if } ~ \ep ~ {\rm is} ~ {\rm small} ~ {\rm enough}.
\ees
Hence, for $x\in [0,\underline{h}(t)/4]$, we have
\bess
&&d\dd\int_{0}^{\underline h(t)}J(x,y)\underline u(t,y)\dy-dj(x)\underline u+ f(\underline u)\ge-d\ep\underline{u}+\rho\sqrt{\ep}\ge-d\ep u^*+\rho\sqrt{\ep}\ge0  \;\; {\rm if } ~ \ep\le\left(\frac{\rho}{du^*}\right)^2,
\eess
and for $x\in [\frac{\underline{h}(t)}{4},\underline{h}(t)]$,
\bes
 &&d\dd\int_{0}^{\underline h(t)}J(x,y)\underline u(t,y)\dy-dj(x)\underline u+ f(\underline u)\nonumber\\
&\ge& d\dd\int_{0}^{\underline h(t)}J(x,y)\underline u(t,y)\dy-\left(d-\rho\ep\right)\underline u\nonumber\\
&=&\left(\min\left\{\frac{\rho\ep}2,d\right\}+\left(d-\frac{\rho\ep}2\right)^+\right)
\int_{0}^{\underline h(t)}J(x,y)\underline u(t,y)\dy-\left(d-\rho\ep\right)\underline u\nonumber\\
&\ge&\min\left\{\frac{\rho\ep}2,d\right\}\hat C l_{\ep}\underline{h}^{1-\gamma}(t)+\left(d-\frac{\rho\ep}2\right)^+(1-\ep^2)\underline{u}
-\left(d-\rho\ep\right)\underline u\nonumber\\
&\ge&\min\left\{\frac{\rho\ep}2,d\right\}\hat C  l_{\ep}\underline{h}^{1-\gamma}(t) ~ ~ ~ {\rm with } ~ \ep ~  {\rm small} ~ {\rm sufficiently}.
\label{x.1}\ees
Besides, we have $\underline{u}_t(t,x)=0$ for $t>0$ and $x\in[0,\frac{\underline{h}(t)}2]$, and
\[\underline{u}_t(t,x)=2 l_{\ep}\frac{x\underline{h}'(t)}{\underline{h}^2(t)}\le 2 l_{\ep}\frac{\underline{h}'(t)}{\underline{h}(t)}=\frac{2 l_1 l_{\ep}}{\gamma-1}\underline{h}^{1-\gamma}(t) ~ ~ {\rm for} ~ t>0, ~ x\in(\frac{\underline{h}(t)}2,\underline{h}(t)).\]
So the first inequality of \eqref{4.3} holds if $\frac{2 l_1}{\gamma-1}\le\min\left\{\frac{\rho\ep}2,d\right\}\hat C$. Moreover, since spreading happens, there is $T>0$ such that
\[\underline{h}(0)\le h(T) ~ ~ {\rm and } ~ ~ \underline{u}(0,x)\le  l_{\ep}=u^*-\sqrt{\ep}\le u(T,x) ~ ~ {\rm for } ~ x\in[0,\underline{h}(0)].\]
The proof is complete.
\end{proof}

\begin{lemma}\label{l4.3}Let the condition ${\bf(J^\gamma)}$ hold with $\gamma=2$ and spreading happen for the problem \eqref{1.2}. Then
\bess
&h(t)\ge C t\ln t ~ ~  {\rm for ~ some ~ constant } ~ C>0 ~ {\rm and} ~ t\gg1,\\
&\dd\liminf_{t\to\yy}u(t,x)\ge u^* ~ ~ {\rm uniformly ~ in} ~ [0,s(t)] ~ ~ {\rm for ~ any } ~ 0\le s(t)=(t\ln t)o(1).
\eess
\end{lemma}

\begin{proof}For the fixed $\alpha\in(0,1)$ and $ l_{\ep}=u^*-\sqrt{\ep}$ with $0<\ep\ll1$, we define
\[\underline{h}(t)= l_1(t+\theta)\ln(t+\theta),\;\;\;\underline{u}(t,x)= l_{\ep}\min\left\{1,\,\frac{\underline{h}(t)-x}{(t+\theta)^{\alpha}}\right\} ~ ~ {\rm for} ~ t\ge0, ~ x\in[0,\underline{h}(t)]\]
with $\theta,\, l_1$ to be determined later. Clearly,
 \[\frac{\underline{h}(t)}{(t+\theta)^{\alpha}}\to\yy ~ ~ {\rm uniformly} ~ t\ge0 ~ ~ {\rm as}~ \theta\to\yy.\]
We will show that there exist suitable $\theta,\,l_1$ and $T>0$ such that
  \bes\label{4.6}
\left\{\begin{aligned}
&\underline u_t\le d\dd\int_{0}^{\underline h(t)}J(x,y)\underline u(t,y)\dy-dj(x)\underline u+ f(\underline u), && t>0,~x\in[0,\underline h(t))\setminus\left\{\underline{h}(t)-(t+\theta)^{\alpha}\right\},\\
&\underline u(t,\underline h(t))\le0,&& t>0,\\
&\underline h'(t)\le\mu\dd\int_{0}^{\underline h(t)}\!\!\int_{\underline h(t)}^{\infty}
J(x,y)\underline u(t,x)\dy\dx,&& t>0,\\
&\underline h(0)\le h(T),\;\;\underline u(0,x)\le u(T,x),&& x\in[0,\underline h(0)].
\end{aligned}\right.
\ees
Similarly to Lemma \ref{l4.2}, we can complete the proof if \eqref{4.6} holds. Firstly, using $J(x,y)=J(x-y)$,
\bess
\mu\dd\int_{0}^{\underline h(t)}\!\!\int_{\underline h(t)}^{\infty}\!
J(x,y)\underline u(t,x)\dy\dx&\ge&\mu l_{\ep}\dd\int_{\frac{\underline{h}(t)}2}^{\underline{h}(t)-(t+\theta)^{\alpha}}\!\int_{\underline h(t)}^{\infty}\!J(x,y)\dy\dx\\
&=&\mu  l_{\ep}\int_{-\frac{\underline{h}(t)}2}^{-(t+\theta)^{\alpha}}\!\int_{0}^{\infty}\!
J(x,y)\dy\dx\\
 &=&\mu  l_{\ep}\int_{-\frac{\underline{h}(t)}2}^{-(t+\theta)^{\alpha}}\!\int_{-x}^{\infty}\!
J(y)\dy\dx\\
&=& l_{\ep}\mu\left(\int_{(t+\theta)^{\alpha}}^{\frac{\underline{h}(t)}2}
\int_{(t+\theta)^{\alpha}}^{y}
J(y)\dx\dy+\int_{\frac{\underline{h}(t)}2}^{\yy}
\int_{(t+\theta)^{\alpha}}^{\frac{\underline{h}(t)}2}
J(y)\dx\dy\right)\\
 &\ge& l_{\ep}\mu\int_{(t+\theta)^{\alpha}}^{\frac{\underline{h}(t)}2}\int_{(t+\theta)^{\alpha}}^{y}
J(y)\dx\dy\\
 &\ge&  l_{\ep}\mu\int_{2(t+\theta)^{\alpha}}^{\frac{\underline{h}(t)}2}
J(y)\left[y-(t+\theta)^{\alpha}\right]\dy\\
&\ge& \frac{ l_{\ep}\mu}2\int_{2(t+\theta)^{\alpha}}^{\frac{\underline{h}(t)}2}
J(y)y\dy\\
 &\ge& \frac{ l_{\ep}\mu \varsigma_1}2\int_{2(t+\theta)^{\alpha}}^{\frac{\underline{h}(t)}2}
y^{-1}\dy\\
&=&\frac{ l_{\ep}\mu \varsigma_1}2\left[\ln(t+\theta)+\ln l_1+\ln\ln(t+\theta)-2\ln2-\alpha\ln(t+\theta)\right]\\
&\ge&\frac{ l_{\ep}\mu \varsigma_1}2(1-\alpha)\ln(t+\theta)
\eess
provided that $\theta$ is large enough. Direct calculations show that
\[\underline{h}'(t)= l_1\ln(t+\theta)+ l_1\le 2 l_1\ln(t+\theta)\le\frac{ l_{\ep}\mu \varsigma_1}2(1-\alpha)\ln(t+\theta) ~ ~ ~ {\rm if}\; ~ \frac{ l_{\ep}\mu \varsigma_1}2(1-\alpha)\ge2 l_1.\]
Hence the third inequality of \eqref{4.6} holds. Now we prove the first inequality of \eqref{4.6}. Obviously, $\underline{u}(t,x)\ge l_{\ep}\frac{\underline{h}-x}{2(t+\theta)^{\alpha}}$ for $x\in[\underline{h}(t)-2(t+\theta)^{\alpha},\underline{h}(t)]$. Thus, for $x\in[\underline{h}(t)-(t+\theta)^{\alpha},\underline{h}(t)]$,
 \bess
\int_{0}^{\underline h(t)}J(x,y)\underline u(t,y)\dy&=&\int_{-x}^{\underline h(t)-x}J(y)\underline u(t,x+y)\dy\\
 &\ge&\frac{ l_{\ep}}2\int_{-(t+\theta)^{\alpha}}^{-(t+\theta)^{\alpha/2}}J(y)
 \frac{\underline{h}(t)-x-y}{(t+\theta)^{\alpha}}\dy\\
&\ge&\frac{ l_{\ep}}2\int_{-(t+\theta)^{\alpha}}^{-(t+\theta)^{\alpha/2}}J(y)
\frac{-y}{(t+\theta)^{\alpha}}\dy\\
 &\ge&\frac{ l_{\ep}\varsigma_1}2\int_{-(t+\theta)^{\alpha}}^{-(t+\theta)^{\alpha/2}}
 \frac{(-y)^{-1}}{(t+\theta)^{\alpha}}\dy\\
&\ge&\frac{ l_{\ep}\varsigma_1\alpha\ln(t+\theta)}{4(t+\theta)^{\alpha}}.
\eess
Moreover, for $x\in[\frac{\underline{h}(t)-(t+\theta)^{\alpha}}2,\underline{h}(t)-(t+\theta)^{\alpha}]$, we have
\bess
 \int_{0}^{\underline h(t)}J(x,y)\underline u(t,y)\dy&=&\int_{-x}^{\underline h(t)-x}J(y)\underline u(t,x+y)\dy\\
  &\ge&  l_{\ep}\int_{-(t+\theta)^{\alpha}}^{-(t+\theta)^{\alpha/2}}J(y)\dy\\
&\ge & l_{\ep}\varsigma_1\int_{-(t+\theta)^{\alpha}}^{-(t+\theta)^{\alpha/2}}(-y)^{-2}\dy\\
&=&\varsigma_1 l_{\ep}\frac{(t+\theta)^{\alpha/2}-1}{(t+\theta)^{\alpha}}\\
 &\ge& \frac{ l_{\ep}\varsigma_1\alpha\ln(t+\theta)}{2(t+\theta)^{\alpha}}
\eess
with $\theta\gg1$. Take $\kappa_2=\underline{h}(t)$ and $\kappa_1=(t+\theta)^{\alpha}$ in Proposition \ref{p4.1} to deduce
\[\int_{0}^{\underline h(t)}J(x,y)\underline u(t,y)\dy\ge(1-\ep^2)\underline{u}(t,x) ~ ~ ~ {\rm for} ~ t>0, ~ x\in\kk[\frac{\underline{h}(t)-(t+\theta)^{\alpha}}2,\,\underline{h}(t)\rr].\]
For $x\in[0,\frac{\underline{h}(t)-(t+\theta)^{\alpha}}2]$, we can argue as in the proof of Lemma \ref{l4.2} to get
\[\int_{0}^{\underline h(t)}J(x,y)\underline u(t,y)\dy\ge(j(x)-\ep)\underline{u}(t,x),\]
and thanks to \eqref{4.4}, we have
\[d\dd\int_{0}^{\underline h(t)}J(x,y)\underline u(t,y)\dy-dj(x)\underline u+ f(\underline u)\ge-\ep du^*+\rho\sqrt{\ep}\ge0 ~ ~ {\rm if } ~ \ep ~ {\rm is} ~ {\rm small} ~ {\rm enough}.\]
For $x\in[\frac{\underline{h}(t)-(t+\theta)^{\alpha}}2,\underline{h}(t)]$, similarly to the  derivation of \eqref{x.1} we can deduce, by using \eqref{4.5},
\bess
 d\dd\int_{0}^{\underline h(t)}J(x,y)\underline u(t,y)\dy-dj(x)\underline u+ f(\underline u)\ge\min\left\{\frac{\rho\ep}2,d\right\}\frac{ l_{\ep}\varsigma_1\alpha\ln(t+\theta)}{4(t+\theta)^{\alpha}} ~ ~ {\rm if } ~ 0<\ep\ll 1.
\eess
On the other hand, we have $\underline{u}_t=0$ for $t>0$ and $x\in[0,\underline{h}(t)-(t+\theta)^{\alpha})$, and
 \[\underline{u}_t=\frac{ l_1 l_{\ep}(1-\alpha)\ln(t+\theta)+ l_1 l_{\ep}}{(t+\theta)^{\alpha}}
 +\frac{ l_{\ep}\alpha x}{(t+\theta)^{1+\alpha}}\le\frac{\ln(t+\theta)}{(t+\theta)^{\alpha}}\left[ l_1 l_{\ep}(1-\alpha)
 + l_1 l_{\ep}+ l_1\right]\]
for $t>0$ and $x\in(\underline{h}(t)-(t+\theta)^{\alpha},\underline{h}(t)]$. Thus, the first inequality of \eqref{4.6} holds if $ l_1$ is suitably small. Moreover, for $\theta$ and $ l_1$ chosen as above, since spreading happens, we can choose $T>0$ such that $\underline{h}(0)\le h(T)$ and $\underline{u}(0,x)\le l_{\ep}=u^*-\sqrt{\ep}\le u(T,x)$. The proof is end.
\end{proof}
Theorem \ref{t4.1} directly follows from Lemmas \ref{l4.1}, \ref{l4.2}, \ref{l4.3} and the fact that $\limsup_{t\to\yy}u(t,x)\le u^*$ uniformly in $\overline{\mathbb{R}}^+$.
\section{Discussion}

This paper is focused on some sharp estimates for model \eqref{1.2}, in which the species is assumed to only enlarge their habitat from right boundary. As for the left boundary $x=0$, we suppose that there is no flux of populations through it, which is analogous to the usual homogenous Neumann boundary condition $\frac{\partial u}{\partial n}=0$. The well-posedness, spreading-vanishing dichotomy and spreading speed have been discussed in \cite{LW21}. Our aims in this paper are the following two aspects.

\sk{\rm(1)}\, In \cite{LW21}, they showed that $\lim_{t\to\yy}u(t,x)=u^*$ locally uniformly in $\overline{\mathbb{R}}^+$ if spreading happens. However, we here prove the more accurate longtime behavior for solution component $u$ of \eqref{1.2}, namely, Theorem \ref{t2.1}, which is also different from that of problem \eqref{1.3} since spreading speed of \eqref{1.3} is always finite and accelerated spreading may occur for \eqref{1.2}. Moreover, the limiting profiles of \eqref{1.2} are investigated.

\sk{\rm(2)}\, For the kernel function with compact support, we give some estimates for spreading speed which seem to be different from those of \cite{DN21} and \cite{DN212}; For the kernel function behaving like $|x|^{-\gamma}$ with $\gamma\in(1,2]$ near infinity, we show that the similar estimates with those in \cite{DN21} and \cite{DN212} for rate of accelerated spreading hold for \eqref{1.2}. Particularly, some interesting conclusions can be drawn from Theorems \ref{t2.1} and  \ref{t4.1}: (1) in the case of accelerated spreading, the population density of species may grow faster than the finite spreading case; (2) if we improve the condition of $J$, such as letting $J$ satisfy ${\bf(J^\gamma)}$, then more specific longtime behaviors can be obtained.

As we see, for local diffusion free boundary problems, sharper estimates have been obtained, for example, please see \eqref{1.4}. However, whether the results in \eqref{1.4} hold for the nonlocal version \eqref{1.1} or \eqref{1.2} is still open. We leave it as an important future work.

\end{document}